\documentclass[11pt]{amsart}
\usepackage[utf8]{inputenc}
\usepackage{fontenc}
\usepackage{amsfonts}
\usepackage{amssymb}
\usepackage{amsmath}
\usepackage{amsthm}\usepackage{mathtools}
\usepackage{enumerate}
\usepackage{hyperref}\usepackage{enumitem}
\usepackage{mathrsfs}
\usepackage{tikz}
\usepackage{marginnote}

\newcommand{\mathscripty}{\mathscr}

\newcommand{\rs}{\mathord{\upharpoonright}}

\newcommand{\e}{\epsilon}

\newcommand{\N}{\mathbb{N}}

\newcommand{\bA}{A}

\newcommand{\cQ}{\mathcal{Q}}
\newcommand{\eps}{\epsilon}

\newtheorem*{OCAT}{$\OCA$}
\newtheorem*{MAkappa}{$\mathbf{MA_\kappa}$}
\theoremstyle{Case1}

\theoremstyle{Case2}

\newcommand{\NN}{\mathbb{N}}

\newcommand{\PP}{\mathbb{P}}

\newcommand{\cstu}{\mathrm{C}^*_u}
\newcommand{\roeq}{\mathrm{Q}^*_u}
\newtheorem*{rigprob*}{Rigidity Problem for Uniform Roe Algebras}
\newtheorem*{rigprobcorona*}{Rigidity Problem for Uniform Roe Coronas}

\newcommand{\SI}{\mathscripty{I}}
\newcommand{\SJ}{\mathscripty{J}}

\newcommand{\Cstar}{\mathrm{C}^*}
\newcommand{\cst}{\mathrm{C}^*}
\newcommand{\cstar}{$\mathrm{C}^*$}

\newcommand{\cP}{\mathcal{P}}
\newcommand{\bbN}{\mathbb{N}}
\newcommand{\bbZ}{\mathbb{Z}}
\newcommand{\cB}{\mathcal{B}}
\newcommand{\cK}{\mathcal{K}}

\newcommand{\fS}{\mathfrak{S}}
\newcommand{\cPNF}{\mathcal{P}(\bbN)/\Fin}

\newcommand{\ZFC}{\mathrm{ZFC}}

\newcommand{\MA}{\mathrm{MA}}

\newcommand{\OCAinf}{\mathrm{OCA_\infty}}
\newcommand{\OCA}{\mathrm{OCA_T}}

\newtheorem{theorem}{Theorem}[section]
\newtheorem*{theorem*}{Theorem}
\newtheorem{proposition}[theorem]{Proposition}
\newtheorem{problem}[theorem]{Problem}
\newtheorem*{proposition*}{Proposition}
\newtheorem{lemma}[theorem]{Lemma}
\newtheorem*{lemma*}{Lemma}
\newtheorem{corollary}[theorem]{Corollary}
\newtheorem*{corollary*}{Corollar}

\newtheorem*{fact*}{Fact}
\theoremstyle{definition}
\newtheorem{definition}[theorem]{Definition}
\newtheorem*{definition*}{Definition}
\newtheorem{claim}[theorem]{Claim}
\newtheorem*{claim*}{Claim}

\newtheorem*{conjecture*}{Conjecture}

\newtheorem{assumption}[theorem]{Assumption}
\newtheorem{question}[theorem]{Question}

\theoremstyle{remark}
\newtheorem{example}[theorem]{Example}
\newtheorem*{example*}{Example}

\newtheorem*{remark*}{Remark}

\newtheorem*{note*}{Note}
\newtheorem*{question*}{Question}
\newtheorem{exmpl}[theorem]{Example}

\newcommand{\norm}[1]{\left\lVert #1 \right\rVert}

\DeclareMathOperator{\Sp}{Sp}

\DeclareMathOperator{\Id}{Id}

\DeclareMathOperator{\supp}{supp}

\DeclareMathOperator{\propg}{prop}

\DeclareMathOperator{\Fin}{Fin}

\DeclareMathOperator{\rank}{rank}

\DeclareMathOperator{\Ad}{Ad}

\DeclareMathOperator{\diam}{diam}

\newcommand{\calD}{\mathcal D} 
 
\newcommand{\cG}{\mathcal G} 
\newcommand{\cM}{\mathcal M} 
\newcommand{\cN}{\mathcal N} 
\newcommand{\bbC}{\mathbb C}

\newcounter{my_enumerate_counter}
\newcommand{\pushcounter}{\setcounter{my_enumerate_counter}{\value{enumi}}}
\newcommand{\popcounter}{\setcounter{enumi}{\value{my_enumerate_counter}}}

\usepackage{enumitem}

\begin{document}

\title{Uniform Roe coronas}%

\author[B. M. Braga]{Bruno M. Braga}
\address[B. M. Braga]{Department of Mathematics and Statistics,
York University,
4700 Keele Street,
Toronto, Ontario, Canada, M3J
1P3}
\email{demendoncabraga@gmail.com}
\urladdr{https://sites.google.com/site/demendoncabraga}

\author[I. Farah]{Ilijas Farah}
\address[I. Farah]{Department of Mathematics and Statistics,
York University,
4700 Keele Street,
Toronto, Ontario, Canada, M3J
1P3}
\email{ifarah@mathstat.yorku.ca}
\urladdr{http://www.math.yorku.ca/$\sim$ifarah}

\author[A. Vignati]{Alessandro Vignati}
\address[A. Vignati]{ 
Institut de Math\'ematiques de Jussieu - Paris Rive Gauche (IMJ-PRG)\\
Universit\'e de Paris\\
B\^atiment Sophie Germain\\
8 Place Aur\'elie Nemours \\ 75013 Paris, France}
\email{ale.vignati@gmail.com}
\urladdr{http://www.automorph.net/avignati}

\subjclass[2010]{}
\thanks{}
\date{\today}%
\maketitle

\begin{abstract}
A uniform Roe corona is the quotient of the uniform Roe algebra of a metric space by the ideal of compact operators. Among other results, we show that it is consistent with ZFC that isomorphism between uniform Roe coronas implies coarse equivalence between the underlying spaces, for the class of uniformly locally finite metric spaces which coarsely embed into a Hilbert space. Moreover, for uniformly locally finite metric spaces with property A, it is consistent with ZFC that isomorphism between the uniform Roe coronas is equivalent to bijective coarse equivalence between some of their cofinite subsets. 
We also find locally finite metric spaces such that the isomorphism of their uniform Roe coronas is independent of ZFC. 
All set-theoretic considerations in this paper are relegated to two `black box' principles.
\end{abstract}

\setcounter{tocdepth}{1}
\tableofcontents
\section{Introduction}\label{SectionIntro}

Given 
 a metric space $(X,d)$, one can define a \cstar-subalgebra $\cstu(X)$ of the space of operators on $\ell_2(X)$ called the \emph{uniform Roe algebra of $X$}. More precisely, $\cstu(X)$ is defined as the norm closure of the 
 algebra of all operators on $\ell_2(X)$ of finite propagation with respect to the metric $d$ (we refer the reader to Section~\ref{SectionPrelim} for
precise definitions). 
For recent results on uniform Roe algebras, as well as Roe algebras, see 
\cite{RoeWillett2014,SpakulaWillett2013Crelle,SpakulaWillett2013AdvMath}. 
 The motivation for the study of these algebras comes from its intrinsic relation with the
coarse Baum-Connes conjecture and, consequently, with the Novikov conjecture
 \cite{Yu2000}. One of the main questions about uniform Roe algebras is whether this \cstar-algebra completely determines the large scale geometry of the underlying metric space.

 \begin{problem}\label{ProblemRigProm}\emph{\textbf{(Rigidity of Uniform Roe Algebras)}}
Let $X$ and $Y$ be metric spaces such that $\cstu(X)$ and $\cstu(Y)$ are isomorphic. Does it follow that $X$ and $Y$ are coarsely equivalent? 	
\end{problem}

Recently, much progress has been made on the rigidity problem within the class of uniformly locally finite metric spaces. Precisely, in \cite[Theorem~1.8]{SpakulaWillett2013AdvMath}, it was shown that Problem \ref{ProblemRigProm} has a positive answer for uniformly locally finite metric spaces with G. Yu's property A.\footnote{Since we do not make explicit use of it, we do not properly define property A here (see \cite[\S11.5]{RoeBook} for its definition).} The first two authors, improved this result in \cite[Corollary 1.2]{BragaFarah2018Trans}. They showed that Problem \ref{ProblemRigProm} has a positive answer for uniformly locally finite metric spaces which coarsely embed into a Hilbert space and for uniformly locally finite spaces such that all ghosts projections are compact. Recall, an operator $a\in \cstu(X)$ is called a \emph{ghost} if for all $\eps>0$ there exists a bounded 
 $F\subset X$ such that $|\langle a\delta_x,\delta_y\rangle|<\eps$ for all $x$ and $y$ in $X\setminus F$.\footnote{A uniformly locally finite metric space $X$ has property A if and only if all ghost operators in $\cstu(X)$ are compact, by \cite[Theorem 1.3]{RoeWillett2014}.}

 Our objective is to understand what information about the uniform Roe algebra is preserved after passing to the quotient by the ideal of compact operators. Towards this end we utilize the theory of liftings for corona \cstar-algebras (\cite{Farah.C,Gha:FDD,McKenney.UHF,MKV.FA,V.PhDThesis,vignati2018rigidity}).

 \begin{definition}
Let $X$ be a countable metric space and $\cstu(X)$ be the uniform Roe algebra of $X$. The \emph{uniform Roe corona} $\roeq(X)$ is defined by 
\[
\roeq(X)=\cstu(X)/\mathcal K(\ell_2(X)),
\]
where $\mathcal K(\ell_2(X))$ denotes the space of all compact operators on $\ell_2(X)$.
\end{definition}

The terminology is justified by the resemblance of these quotient structure to corona \cstar-algebras. The rigidity problem has a clear version for uniform Roe coronas.

\begin{problem}\label{ProblemRigProbCorona}\emph{\textbf{(Rigidity of Uniform Roe Coronas)}}
Let $X$ and $Y$ be metric spaces such that $\roeq(X)$ and $\roeq(Y)$ are isomorphic. Does it follow that $X$ and $Y$ are coarsely equivalent? 	
\end{problem}

The study of isomorphisms between quotient algebras is intrinsically related to the search for liftings of those isomorphisms. Our notes deal with the following notion of lift.

\begin{definition}\label{Def.LiftOnDiagIso}
Let $X$ and $Y$ be countable metric spaces. 
\begin{enumerate}
\item A $^*$-homomorphism $\Lambda\colon \roeq(X)\to \roeq(Y)$ is \emph{liftable on the diagonal} if there is a strongly continuous $^*$-homomorphism $\Phi\colon \ell_\infty(X)\to \cstu(Y)$ which lifts $\Lambda$ on $\ell_\infty(X)/c_0(X)$ (see Definition~\ref{def:lift}). 

\item If $\Lambda$ is an isomorphism, we say that it is \emph{liftable on diagonals} if both~$\Lambda$ and $\Lambda^{-1}$ are liftable on the diagonal. 
 If this is the case then we say that $\roeq(X)$ and $\roeq(Y)$ are \emph{liftable on diagonals isomorphic}.
\end{enumerate}
\end{definition}

In order to guarantee that an automorphism is liftable on diagonals it is often necessary to work within a theory that extends the standard Zermelo--Fraenkel axioms for set theory, ZFC (see \cite{Farah.C}, \cite{MKV.FA}, and Example~\ref{Ex:RudinShelah}). 
 Our rigidity results are relatively consistent with ZFC and hold under the Open Coloring Axiom, $\OCA$, and Martin's Axiom, $\mathrm{MA}_{\aleph_1}$ (see \S\ref{SubsectionSetTheoryForcingAxiom} for definitions). 
These axioms are used only indirectly, via a `black box' extracted from \cite{MKV.FA}. The following is proved in \S\ref{S:proof.blackbox}.

\begin{theorem} \label{T:blackbox}
Assume $\OCA$ and $\MA_{\aleph_1}$. Then every isomorphism 
between uniform Roe coronas of countable metric spaces is liftable on the diagonals. 
\end{theorem}

Before stating our rigidity results, we need two definitions regarding the geometry of metric spaces. 

\begin{definition}
Let $(X,d)$ be a metric space.
\begin{enumerate}
\item The space $X$ is \emph{sparse} if there exists a partition $X=\bigsqcup_nX_n$ of $X$ into finite subsets such that $d(X_n,X_m)\to \infty$ as $n+m\to \infty$. 
\item The space $X$ \emph{yields only compact ghost projections} if every ghost projection in $\cstu(X)$ is compact.
\end{enumerate} 
\end{definition}

Any locally finite metric space whose sparse subspaces yield only compact ghost projections is uniformly locally finite (Lemma~\ref{L.ulf.ghost}). 
Each of the conditions `$X$ yields only compact ghost projections' 
and `$X$ is coarsely embeddable into a Hilbert space' 
separately implies that 
sparse subspaces of $X$ yield only compact ghost projections 
(for the latter, see \cite[Lemma 7.3]{BragaFarah2018Trans}). 

The following is one of our main results, proved in \S\ref{SectionCoarseEquiv}.

\begin{theorem} \label{ThmRigidityRoeCorona}
Suppose $X$ and $Y$ are locally finite metric spaces such that all of their sparse subspaces yield only compact ghost projections. 
 If $\roeq(X)$ and $\roeq(Y)$ are liftable on diagonals isomorphic, then 
 $X$ and $Y$ are coarsely equivalent.
\end{theorem}

By Theorem~\ref{T:blackbox}, we have the following. 

\begin{corollary}\label{Cor.ThmRigidityRoeCorona} Assume  $\OCA$ and $\mathrm{MA}_{\aleph_1}$. 
Suppose $X$ and $Y$ are locally finite metric spaces all of whose sparse subspaces yield only compact ghost projections. If $\roeq(X)\cong \roeq(Y)$, then $X$ and $Y$ are coarsely equivalent.\qed 
\end{corollary}


It was proved in \cite[Corollary~6.13]{WhiteWillett2017} that the existence of an 
isomorphism between $\cstu(X)$ and $\cstu(Y)$ is equivalent to bijective coarse equivalence for uniformly locally finite metric spaces $X$ and $Y$ with property A. The next result is a version of that for isomorphisms between uniform Roe coronas which are liftable on the diagonals. 

\begin{theorem} \label{ThmRigidityRoeCoronaPropertyA} Let $X$ and $Y$ be uniformly locally finite metric spaces and assume that $X$ has property A.
The following are equivalent.
\begin{enumerate}
\item\label{ItemLiftIso} The uniform Roe coronas $\roeq(X)$ and $\roeq(Y)$ are liftable on diagonals isomorphic.
\item\label{ItemEquivv} There exist cofinite subsets $\tilde X\subseteq X$ and $\tilde Y\subseteq Y$ such that $\tilde X$ and $\tilde Y$ are bijectively coarsely equivalent.
\end{enumerate}
\end{theorem}


By Theorem~\ref{T:blackbox}, we have the following. 

\begin{corollary} \label{C:ThmRigidityRoeCoronaPropertyA} 
Assume $\OCA$ and $\mathrm{MA}_{\aleph_1}$.
Let $X$ and $Y$ be uniformly locally finite metric spaces and assume that $X$ has property A. The following are equivalent.
\begin{enumerate}
\item The uniform Roe coronas $\roeq(X)$ and $\roeq(Y)$ are isomorphic.
\item There exist cofinite subsets $\tilde X\subseteq X$ and $\tilde Y\subseteq Y$ such that $\tilde X$ and $\tilde Y$ are bijectively coarsely equivalent. \qed
\end{enumerate}
\end{corollary}

By translating two results of S. Ghasemi (\cite[Theorem~1.2]{Ghasemi.FFV} and \cite{Gha:FDD}) into the language of uniform Roe coronas, and dropping uniform local finiteness, we obtain the following independence result.

\begin{theorem}\label{T.Independent}
There are locally finite metric spaces $X$ and $Y$ such that the assertion $\roeq(X)\cong \roeq(Y)$
is independent from ZFC. 
\end{theorem}

This is a consequence of Theorem~\ref{T.Independent+}, where we construct a large family of spaces with this property. (We should note that the spaces constructed in Theorem~\ref{T.Independent} are 
not uniformly locally finite, and $\cstu(X)$ and $\cstu(Y)$ have noncompact ghost projections.)

Similarly, classical results of W. Rudin and S. Shelah imply that 
there exists a uniformly locally finite metric space~$X$ such that the 
assertion `Every automorphism of $\roeq(X)$ is liftable on the diagonals'
is independent from ZFC (Example~\ref{Ex:RudinShelah}).

The paper is organized as follows. In \S\ref{SectionPrelim}, we present all the notation and terminology needed for these notes. In particular, in \S\ref{SubsectionSetTheoryForcingAxiom}, we present the set theoretical axioms $\OCA$ and $\mathrm{MA}_{\aleph_1}$ as well as Theorem \ref{ThmAlessandroPaulHomomorphism}, which is our main tool in order to obtain Corollary \ref{Cor.ThmRigidityRoeCorona} and Corollary \ref{C:ThmRigidityRoeCoronaPropertyA}. In \S\ref{SectionCoarseLike}, we show that the liftings obtained by Theorem \ref{ThmAlessandroPaulHomomorphism} are coarse-like (see Definition~\ref{Def.CoarseLike+} below), and \S\ref{SectionGeomCond} is dedicated to the technical lemmas which depend on the geometric properties of our metric spaces. Theorem \ref{ThmRigidityRoeCorona} and Theorem \ref{ThmRigidityRoeCoronaPropertyA} are proved in \S\ref{SectionCoarseEquiv} and \S\ref{SectionBijCoarseEquiv}, respectively. At last, in \S\ref{sec:counterexamples}, we construct a class of locally finite metric spaces for which the existence of isomorphisms between their uniform Roe coronas is independent from ZFC.

\section{Preliminaries}\label{SectionPrelim}

\subsection{Uniform Roe algebras and uniform Roe coronas}
Given a complex Hilbert space $H$, $\cB(H)$ denotes the space of bounded operators on $H$, and $\cK(H)$ the space of compact operators on $H$. If $X$ is a set, $\ell_2(X)$ is the complex Hilbert space of square summable sequences indexed by $X$, with canonical basis $\{\delta_x\}_{x\in X}$. Denote by $\pi_X$ the canonical quotient map \allowbreak $\pi_X\colon \cB(\ell_2(X))\to \cB(\ell_2(X))/\cK(\ell_2(X))$. The \emph{support of} $a\in\cB(\ell_2(X))$ is defined as 
\[
\supp(a)=\{(x,y)\in X\times X\colon\langle a\delta_x, \delta_y\rangle\neq 0\}.
\] 
Writing $\Delta_X=\{(x,x)\colon x\in X\}$, 
the algebra $\ell_\infty(X)$ is naturally identified with the subalgebra $\{a\in\cB(\ell_2(X))\colon \supp(a)\subseteq \Delta_X)\}$. Given $x,y\in X$, denote by $e_{yx}$ the operator in $\cB(\ell_2(X))$ given by 
\[e_{yx}(\delta_z)=\langle \delta_z,\delta_x\rangle\delta_y,\]
for all $z\in X$. Given $A\subseteq X$, write $\chi_A=\sum_{x\in A}e_{xx}$, so $\chi_A\in \ell_\infty(X)$.

If $X$ is a set and $X'\subseteq X$, we identify $\cB(\ell_2(X'))$ with a subalgebra of $\cB(\ell_2(X))$ in the natural way. If $(X_n)_n$ is a sequence of disjoint subsets of $X$, $\prod_n\cB(\ell_2(X_n))$ is identified with a subalgebra of $\cB(\ell_2(X))$.

If $(X,d)$ is a metric space, we say that $a\in\cB(\ell_2(X))$ \emph{has propagation at most $r$}, and write $\propg(a)\leq r$, if $d(x,y)\leq r$ whenever $(x,y)\in \supp(a)$.

\begin{definition}
Let $X$ be a countable metric space. The \emph{algebraic uniform Roe algebra} $\cstu[X]$ is the subalgebra of $\cB(\ell_2(X))$ of all operators of finite propagation. The \emph{uniform Roe algebra of} $X$, $\cstu(X)$, is the norm closure of $\cstu[X]$ in $\cB(\ell_2(X))$. The \emph{uniform Roe corona} is 
\[
\roeq(X)=\cstu(X)/\mathcal K(\ell_2(X)).
\]
\end{definition}

It is clear that $\propg(a)=0$ for all metric spaces $X$ and all $a\in \ell_\infty(X)$, so $\ell_\infty(X)\subseteq \cstu(X)$. In particular, $\chi_A\in \cstu(X)$, for all $A\subseteq X$. Notice that $\propg(\chi_Ab\chi_A)\leq\propg(b)$ for all $A\subseteq X$ and $b\in\cB(\ell_2(X))$. We will use this fact without any further mention.

Given a countable set $X$, $F\subseteq X$, and $a=(a_x)_{x\in X}\in \ell_\infty(X)$, an element $a\rs F\in\ell_\infty(X)$ is defined by 
\[
(a\restriction F)_x=\begin{cases}
a_{x},& \text{if }x\in F,\\
0,&\text{ otherwise.}
\end{cases}
\]
We identify 
$\ell_\infty(F)$ with the \cstar-subalgebra $\{a\in \ell_\infty(X)\colon a\rs F=a\}$
of $\ell_\infty(X)$.

\subsection{Coarse geometry of metric spaces}

Let $(X,d)$, $(Y,\partial)$ be metric spaces and $f\colon X\to Y$. A function $f$ is \emph{coarse} if 
\[
\sup\{\partial(f(x),f(y))\colon d(x,y)\leq t\}<\infty,
\]
for all $t\geq 0$, and $f$ is called \emph{expanding} if 
\[\lim_{t\to \infty}\inf\{\partial(f(x),f(y))\colon d(x,y)\geq t\}=\infty.\]
We say that $f$ is a \emph{coarse embedding} if it is both coarse and expanding. If~$Z$ is a set and $f,g\colon Z\to X$ are maps, we say that $f$ and $g$ are \emph{close} if \[\sup_{z\in Z}d(f(z),g(z))<\infty.\]
Two metric spaces $(X,d)$ and $(Y,\partial)$ are said to be \emph{coarsely equivalent} if there exist coarse functions $f\colon X\to Y$ and $g\colon Y\to X$ such that $g\circ f$ is close to $\Id_X$ and $f\circ g$ is close to $\Id_Y$. Notice that this automatically implies that $f$ and $g$ are expanding. The maps $f$ and $g$ are called \emph{coarse inverses} of each other.

 If there exists a bijection $f\colon X\to Y$ such that both $f$ and $f^{-1}$ are coarse, then $(X,d)$ and $(Y,\partial)$ are said to be \emph{bijectively coarsely equivalent}.

 For a metric space $(X,d)$, $x\in X$ and $r\geq 0$, denote by $B_r(x)$ the $d$-ball centered at $x$ of radius $r$. The metric space $(X,d)$ is said to be \emph{locally finite} if, for all $r\geq 0$ and all $x\in X$, $|B_r(x)|<\infty$, and \emph{uniformly locally finite} if $\sup_{x\in X}|B_r(x)|<\infty$, for all $r\geq 0$. Clearly, every locally finite metric space is countable and, if infinite, unbounded. 

The following simple lemma was promised in the introduction. 

\begin{lemma} \label{L.ulf.ghost} 
If $X$ is locally finite and 
every sparse subspace of $X$ yields only compact ghost projections, then $X$ is
uniformly locally finite. 
\end{lemma} 

\begin{proof} Suppose $X$ is not uniformly locally finite. Fix $r<\infty$ such that 
there are $x_n\in X$ satisfying $|B_r(x_n)|\geq n$. Let $X_n=B_r(x_n)$. 
Since $X$ is locally finite, by going to a subsequence we can assure that $d(X_m,X_n)\geq m+n$ for all 
$m\neq n$. Then $\tilde X=\bigcup_n X_n$ is a sparse subspace of $X$. 

Consider the rank 1 projection $p_n$ onto the constant functions in $\ell_2(X_n)$. 
 Its propagation is at most $r$ and $\langle p_n\delta_x, \delta_{x'}\rangle =1/|X_n|$ for all $x$ and $x'$ in $X_n$. 
Therefore $\tilde X$ yields a noncompact ghost projection $\sum_n p_n$. 
\end{proof} 

\subsection{Set theory: forcing axioms}\label{SubsectionSetTheoryForcingAxiom}
In this subsection we state the additional set-theoretic 
axioms used in the proof of Theorem~\ref{T:blackbox}. 
These axioms will not be used directly in the present paper, and the reader can skip ahead to the 
next subsection. 

We will now state Todor\v{c}evi\'c' Open Colouring Axiom ($\OCA$), introduced in \cite{Todorcevic.PPIT}. 
If $\mathcal X$ is a set, $[\mathcal X]^2$ denotes the set of unordered pairs of elements of $\mathcal X$.
Subsets of $[\mathcal X]^2$ are identified with symmetric subsets of $\mathcal X^2\setminus \Delta_{\mathcal X}$, thus giving meaning to the phrase `an open subset of $[\mathcal X]^2$' in case $\mathcal X$ is a topological space.

\begin{OCAT}
Let $\mathcal X$ be a separable metric space. If $[X]^2=K_0\sqcup K_1$ where $K_0$ is open then one of the following applies. 
\begin{enumerate}
 \item there is an uncountable $Y\subseteq X$ such that $[Y]^2\subseteq K_0$, or
 \item there are sets $X_n\subseteq X$, for $n\in\N$, such that $X=\bigcup_n X_n$ and $[X_n]^2\subseteq K_1$ for all $n\in\N$.
\end{enumerate}
\end{OCAT}

$\OCA$ is a consequence of the Proper Forcing Axiom (\cite{Sh:PIF}, \cite[p. 382]{Ku:Set}) and it is relatively consistent with $\ZFC$. In a previous version of this manuscript, we used of an apparently stronger version of $\OCA$ known as $\OCAinf$, introduced  in~\cite{Farah.Cauchy}. Recently $\OCAinf$ was shown to be a consequence of $\OCA$ (see \cite{Moore.OCA}, also \cite[Theorem~8.6.6]{Fa:Combinatorial}). This in particular implies that $\OCAinf$ can be replaced with $\OCA$ in the assumptions of  \cite[Theorem 9.4]{MKV.FA}. 

We proceed to state Martin's Axiom after some forcing terminology (see~\cite{Ku:Set} for many more details). 
\begin{definition} \label{Def:Poset}
Two elements $p$ and $q$ of a partial order $(\PP,\leq)$ are \emph{compatible} if there exists $r\in \PP$ such that $r\leq p$ and $r\leq q$, and 
\emph{incompatible} otherwise. 
 A partial order $(\PP,\leq)$ is said to have the \emph{countable chain condition} (\emph{ccc}) if there is no uncountable set of pairwise incompatible elements in $\PP$. A set $D\subseteq \PP$ is \emph{dense} if $\forall p\in\PP\exists q\in D$ with $q\leq p$. A subset $G$ of $\PP$ is a \emph{filter} if it is upward closed and for any $p,q\in G$, there is some $r\in G$ such that $r\le p$ and $r\le q$. 
\end{definition} 

The following is Martin's Axiom at the cardinal $\kappa$.

\begin{MAkappa}
For every poset $(\PP,\leq)$ that has the ccc, and every family of dense subsets $D_\alpha\subseteq \PP$ ($\alpha < \kappa$), there is a filter $G\subseteq \PP$ such that $G\cap D_\alpha\neq\emptyset$ for every $\alpha < \kappa$.
\end{MAkappa}
\noindent$\MA_{\aleph_0}$ is a theorem of $\ZFC$ (it is a close relative to the Baire Category Theorem), as is the negation of
$\MA_{2^{\aleph_0}}$.

Both $\OCA$ and $\MA_{\aleph_1}$ are consequences of Shelah's Proper Forcing Axiom, PFA. 
Unlike PFA, their relative consistency with $\ZFC$ does not require any large cardinal assumptions (see \cite[p. 376]{Ku:Set}). Each of these axioms contradicts the Continuum Hypothesis. In operator algebras, forcing axioms have been used to imply rigidity phenomena for isomorphisms of corona \cstar-algebras (see \cite[\S 7.1]{Fa:Logic}, \cite{McKenney.UHF}, \cite{MKV.FA},  \cite{vignati2018rigidity}). Notably, $\OCA$ implies that all automorphisms of the Calkin algebra are inner (\cite[Theorem 1]{Farah.C}). On the other hand, the Continuum Hypothesis implies the existence of an outer automorphism (see \cite[Theorem~2.4]{Phillips-Weaver}).

\subsection{Liftings and nonmeager ideals}\label{SubsectionIdealsLiftings} Our proof of Theorem~\ref{T:blackbox} 
proceeds in two stages: First by finding maps that lift a given $^*$-homomorphism on a `large' set, and second, by showing that these maps actually lift it everywhere. In this
subsection we isolate this largeness property. 

If $X$ is a set then $\SI\subseteq \mathcal P(X)$ is an \emph{ideal} if it is closed under subsets and finite unions. (In other words, it is an ideal of the Boolean ring $\cP(X)$.) An ideal $\SI\subseteq \cP(X)$ is \emph{dense} if it contains all finite subsets of $X$ and for every infinite $S\subseteq X$ there is an infinite $T\subseteq S$ with $T\in \SI$.
This is easily proved to be equivalent to the conjunction of $\SI$ being dense in the poset $(\{X'\subseteq X\colon X'$ is infinite$\},\subseteq)$ in the sense of Definition~\ref{Def:Poset},  and $\SI$ being dense in the Cantor set topology on $\cP(X)$ -- i.e., $\cP(X)$ is identified with $2^X$. If $X$ is countable, this topology on $\cP(X)$ is compact and metric, and we can talk about the topological properties of certain ideals. 

A proof of the following classical result can be found, e.g., in \cite[\S3.10]{FarahBook2000}. 

\begin{proposition}[Jalali--Naini, Talagrand] \label{T:J-NT}
Suppose $X$ is countable and $\SI\subseteq \mathcal P(X)$ is an ideal containing all finite subsets of $X$. Then~$\SI$ 
is nonmeager if and only if for every sequence $\{I_n\}$ of disjoint finite subsets of $X$ 
 there is an infinite $L\subseteq\bbN$ such that $\bigcup_{n\in L} I_n\in\SI$. In particular, if $\SI$ is nonmeager, then $\SI$ is dense. \qed
 \end{proposition}

 The following definitions are essential for our approach.

\begin{definition}\label{def:lift}
Suppose $A$ and $B$ are \cstar-algebras, $I$ and $J$ are two-sided, norm-closed, self-adjoint
ideals of $A$ and $B$, respectively, and $\pi_I$ and $\pi_J$ are the corresponding quotient maps.
 If $\Lambda\colon A/I\to B/J$ is a $^*$-homomorphism 
 then a map 
$\Phi\colon A\to B$ is said to \emph{lift} $\Lambda$ if 
$\Lambda(\pi_I(a))=\pi_J(\Phi(a))$ for all $a\in A$ (i.e., if 
the diagram in Fig.~\ref{Fig.Lift} commutes). 
 \begin{figure}[h]
\begin{tikzpicture}
\matrix[row sep=1cm, column sep=1.5cm]
{\node (An) {$A$} ;
& \node (B) {$B$}; \\
\node (An/) {$A/I$} ;
&\node (B/J) {$B/J$};
\\
}; 
\draw (An) edge [->] node [above] {$\Phi$} (B);
\draw (An/) edge [->] node [above] {$\Lambda$} (B/J);
\draw (An) edge [->] node [left] {$\pi_I$} (An/);
\draw (B) edge [->] node [right] {$\pi_J$} (B/J);
\end{tikzpicture}
\caption{\label{Fig.Lift} The map $\Phi$ lifts $\Lambda$.}
\end{figure}
If $Z\subseteq A$ and the equality $\Lambda(\pi_I(a))=\pi_J(\Phi(a))$ 
holds for all $a\in Z$, we say that $\Phi$ \emph{lifts $\Lambda$ on $Z/I$}. 
\end{definition} 
The Axiom of Choice implies that every $\Lambda$ is lifted by some $\Phi$. 
We are interested in the existence of lifts with additional algebraic or topological properties, such 
as being a $^*$-homomorphism or (in the case when $A$ and $B$ are
the multiplier algebras of $I$ and $J$, respectively) being continuous with respect to the strict topologies associated with 
$I$ and $J$.

For any $X$, identifying $\ell_\infty(X)$ with a \cstar-subalgebra of $\cB(\ell_2(X))$, it follows that $\ell_\infty(X)\cap \cK(\ell_2(X))=c_0(X)$. So, we can identify $\ell_\infty(X)/c_0(X)$ with a \cstar-subalgebra of the Calkin algebra $\cB(\ell_2(X))/\cK(\ell_2(X))$. In particular, if~$X$ is a metric space, $\ell_\infty(X)/c_0(X)$ is a \cstar-subalgebra of~$\roeq(X)$.

\begin{definition}\label{almost.Def.LiftOnDiagIso}
Let $X$ and $Y$ be countable metric spaces.
A $^*$-ho\-mo\-mor\-phi\-sm
 $\Lambda\colon \roeq(X)\to \roeq(Y)$ is \emph{almost liftable on the diagonal} if there is  a nonmeager ideal $\SI$ on $X$ such that some  strongly continuous $^*$-homomorphism $\Phi\colon \ell_\infty(X)\to \cB(\ell_2(Y))$  lifts $\Lambda$ on $\{\pi_X(\chi_S)\colon S\in \SI\}$. 
\end{definition}

Note that it is not required that the range of $\Phi$ is included in $\cstu(Y)$. On the other hand this will always be the case (see Proposition~\ref{prop:stayinside}).
We can now state the lifting theorem which will play a crucial role throughout these notes.

\begin{theorem}\label{ThmAlessandroPaulHomomorphism}
Assume $\OCA$ and $\mathrm{MA}_{\aleph_1}$. Suppose $X$ and $Y$ are countable metric spaces. Then every unital injective $^*$-homomorphism 
 of $\Lambda\colon \roeq(X)\to \roeq(Y)$ is almost liftable on the diagonal. 
 \end{theorem}

\begin{proof} This is a special case of \cite[Theorem~9.4]{MKV.FA} when 
$k(n)=1$ for all $n$, $A=\cK(\ell_2(Y))$, and $\roeq(Y)$ is identified with a subalgebra of the Calkin algebra $\cB(\ell_2(Y))/\cK(\ell_2(Y))$.
\end{proof} 

We should point out that Theorem~\ref{ThmAlessandroPaulHomomorphism} 
is a very special case of \cite[Theorem~9.4]{MKV.FA}. The proof of the former 
can also be extracted from the (much easier) proof of \cite[Proposition~7.1]{Farah.C}, where this was proved
under the additional (unnecessary) assumption that $\Lambda$ is the restriction of an automorphism
of the Calkin algebra.


\section{Coarse-like property} \label{SectionCoarseLike}

In this section $X$ will always denote a countable metric space. 
The main result of the present section, Proposition~\ref{prop:stayinside}, shows 
that the lifts $\ell_\infty(X)\to \cstu(Y)$ provided by 
Theorem~\ref{ThmAlessandroPaulHomomorphism} are coarse-like (see below).
The following is \cite[Definition~4.3]{BragaFarah2018Trans}.

\begin{definition} \label{Def.CoarseLike} 
Let $\eps>0$, and $k\in\N$. An operator $a\in \cB(\ell_2(X))$ can be \emph{$\eps$-$k$-approximated} if there exists $b\in\cB(\ell_2(X))$ with propagation at most $k$ such that $\|a-b\|\leq\eps$. We say that $b$ is an \emph{$\eps$-$k$-approximation} of $a$. 
\end{definition} 
With this definition, an operator $a\in\cB(\ell_2(X))$ is in $\cstu(X)$ if and only if for all $\epsilon>0$ there is $k\in\N$ such that $a$ can be $\epsilon$-$k$-approximated.

The following definition was already implicit in \cite[Theorem 4.4]{BragaFarah2018Trans}.

\begin{definition}\label{Def.CoarseLike+}
Let $X$ and $Y$ be metric spaces, $\bA\subset \cstu(X)$ and $\Phi\colon \bA\to \cstu(Y)$ be a map. We say that $\Phi$ is \emph{coarse-like} if for all $m\in\N$ and all $\eps>0$ there exists $k\in \N$ such that $\Phi(a)$ can be $\eps$-$k$-approximated for every contraction $a\in \bA$ with $\propg(a)\leq m$.
\end{definition}

By \cite[Theorem 4.4]{BragaFarah2018Trans} every isomorphism $\Phi\colon \cstu(X)\to \cstu(Y)$ is coarse-like.
The proof of the following proposition is inspired by \cite[Theorem 4.4]{BragaFarah2018Trans}.

\begin{proposition}\label{prop:stayinside}
Let $X$ and $Y$ be countable metric spaces. Suppose that $\Phi\colon\ell_\infty(X)\to\cB(\ell_2(Y))$ is a strongly continuous $^*$-homomorphism which lifts a $^*$-homomorphism $\Lambda$ between the uniform Roe coronas of $X$ and $Y$ on a nonmeager ideal $\SI\subset \cP(X)$ containing all finite subsets of $X$. Then $\Phi$ is coarse-like and the 
image of~$\Phi$ is contained in $\cstu(Y)$.
\end{proposition}

 Before proving Proposition \ref{prop:stayinside}, we need a lemma. This lemma will also be 
 important in \S\ref{S:proof.blackbox} 
and \S\ref{SectionCoarseEquiv}.

\begin{lemma} \label{L:AA} Suppose $X$ and $Y$ are countable 
metric spaces 
 and $\Phi\colon \ell_\infty(X)\to \cstu(Y)$ is 
 a strongly continuous $^*$-homomorphism. 
Then for every $b\in \cK(\ell_2(Y))$ and 
every $\e>0$ there exists a finite $F\subseteq X$ such that 
for all $a\in \ell_\infty(X)$ with $\|a\|\leq 1$ we have 
\[
\max\Big(\norm{b\Phi(a\rs (X\setminus F))},\norm{\Phi(a\rs (X\setminus F))b}\Big)<\e. 
\]
\end{lemma}

\begin{proof} Suppose the conclusion fails for some $b\in \cK(\ell_2(Y))$ and $\eps>0$. Without loss of generality, say $\|b\|=1$. Let $(X_n)_n$ be a sequence of finite subsets of $X$ such that $X=\bigcup_nX_n$ and $X_n\subset X_{n+1}$ for all $n\in\N$. Pick a sequence of contractions $(a_n)_n$ in $\ell_\infty(X)$ such that $a_n$ belongs to $\cB(\ell_2(X\setminus X_n))$ and \[
\max\Big(\norm{b\Phi(a_n)},\norm{\Phi(a_n)b}\Big)\geq \e, 
\]
for all $n\in\N$. Without loss of generality, by going to a subsequence, we can assume that $\norm{b\Phi(a_n)}\geq\eps$ for all $n\in\N$. 

Since $b$ is compact, pick a finite $E\subseteq X$ such that $\|\chi_Eb\chi_E -b\|<\e/2$. So, 
\[\norm{\chi_E\Phi(a_na^*_n)\chi_E}=\norm{\chi_E\Phi(a_n)}^2\geq \norm{\chi_Eb\chi_E\Phi(a_n)}^2\geq\eps^2/4\] for all $n\in\N$. For each $n\in\N$, let $c_n=a_na^*_n$, so $c_n\in\cB(\ell_2(X\setminus X_n))$. Since $E$ is finite and $\Phi$ is strongly continuous, by going to a further subsequence, assume that $\norm{\chi_E\Phi(c_n\rs X_{n+1})\chi_E}\geq\eps^2/8$ for all $n\in\N$.
Define $c\in \ell_\infty(X)$ by letting 
\[c(i)=\begin{cases}
c_n\rs X_{n+1}(i),& \text{ if } n\in\N\text{ and }i\in X_{n+1}\setminus X_n,\\
0, &\text{ if }i\in X_0.
\end{cases}\] 
So, $\sum_{n}c_n\rs X_{n+1}$ converges in the strong operator topology to $c$. Since $E$ is finite and $\Phi$ is strongly continuous, the sum $\sum_{n}\chi_E\Phi(c_n\rs X_{n+1})\chi_E$ converges in norm to $\chi_E\Phi(c)\chi_E$; contradiction, since $\inf_n\norm{\chi_E\Phi(c_n\rs X_{n+1})\chi_E}>0$.
\end{proof}

\begin{proof}[Proof of Proposition \ref{prop:stayinside}]
Since $(X,d)$ is a countable metric space, in order to simplify notation, assume that $X=\N$ and that $d$ is an arbitrary metric on~$\N$. Suppose that the conclusion fails. Fix $\e>0$ such that 
for every $m$ there exists $a\in \ell_\infty(X)$ with $\|a\|\leq 1$ such that 
$\Phi(a)$ cannot be $\e$-${m}$-approximated. 


\begin{claim} For every finite $F\subseteq X$ and every $m\in \bbN$ there is a contraction $a\in \ell_\infty(X\setminus F)$
such that $\Phi(a)$ cannot be $\e/2$-$m$-approximated. Moreover, $a$ can be chosen to have finite support. 
\end{claim} 

\begin{proof} Suppose otherwise, and fix offenders, say $m\in\N$ and a finite $F\subset X$. 
Since $F$ is finite, the unit ball of $\ell_\infty(F)$ is compact. By the metric space instance of 
\cite[Lemma~4.8]{BragaFarah2018Trans} there exists $m$ such that $\Phi(a)$ can be $\e/2$-$m$-approximated for every contraction $a\in \ell_\infty(F)$. But every contraction $b\in \ell_\infty(X)$ can be written as $b_F+b_{X\setminus F}$
for contractions $b_F\in \ell_\infty(F)$ and $b_{X\setminus F} \in \ell_\infty(X\setminus F)$. Therefore, each of $\Phi(b_F)$ and $\Phi(b_{X\setminus F})$ 
can be $\e/2$-$m$-approximated, and \cite[Lemma~4.5]{BragaFarah2018Trans} 
implies that $\Phi(b)$ can be $\e$-$m$-approximated. Since $b\in \ell_\infty(X)$ was arbitrary, 
we have a contradiction. 

We now prove the
 second statement in the claim. Let $a\in \ell_\infty(X\setminus F)$ such that $\Phi(a)$ cannot be $\e/2$-$m$-approximated and let $(X_n)_n$ be a sequence of finite subsets of $X$ such that $X=\bigcup_nX_n$ and $X_n\subset X_{n+1}$ for all $n\in\N$. So, $\lim_n \Phi(a \chi_{X_n})=\Phi(a)$
in the strong operator topology. By 
 \cite[Lemma~4.7]{BragaFarah2018Trans}, there is $n\in\N$ 
such that $\Phi(a\chi_{X_n})$ cannot be $\e/2$-$m$-approximated.
\end{proof} 

By the claim above, there exist a sequence $(F_m)_m$ of disjoint finite subsets of $X$ and a sequence of contractions $(a_m)_m$ such that $a_m\in \ell_\infty(F_m)$ and $\Phi(a_m)$ cannot be $\eps$-$m$-approximated for all $m\in\N$. Since $\Phi$ lifts $\Lambda$ on a nonmeager ideal $\SI\subseteq \cP(\N)$, by 
Theorem~\ref{T:J-NT} there exists an infinite $S\subseteq \bbN$ such that 
for all $L\subseteq S$ we have 
$\bigcup_{n\in L} F_n\in \SI$, 
and therefore $\pi_Y(\Phi(\sum_{n \in L} a_n))\in \roeq(Y)$. 
Since all compact operators on $\ell_2(Y)$ belong to $\cstu(Y)$, 
this implies $\Phi(\sum_{n \in L} a_n)\in \cstu(Y)$ for all $L\subseteq S$. Since $\Phi$ is strongly continuous, $\sum_{n\in L} \Phi(a_n)$ strongly converges to $\Phi(\sum_{n \in L} a_n)\in\cstu(Y)$ for all $L\subseteq S$. 

Let $(Y_n)_n$ be a sequence of finite subsets of $Y$ such that $Y=\bigcup_nY_n$ and $Y_n\subset Y_{n+1}$ for all $n\in\N$. By Lemma \ref{L:AA}, we can pick an infinite $L\subseteq S$ such that $\|\chi_{Y_n}\Phi(a_m)\|\leq 2^{-n-m}$ for all $n<m$ in $L$. Since each $a_m$ is compact, each $\Phi(a_m)$ is compact as well. So, going to a further infinite subset of $S$ if necessary, assume that $\|\chi_{Y_m}\Phi(a_n)- \Phi(a_n)\|\leq 2^{-n-m}$ for all $n\leq m$ in $L$. 

Let $a=\sum_{m\in L} a_m$. Then $\pi_Y(\Phi(a))=\Lambda(\pi_Y(a))$ and $\Phi(a)\in \cstu(Y)$. Fix an operator $c\in \cstu[Y]$ of finite propagation such that $\|\Phi(a)-c\|<\e/2$. Then
\[
\|\Phi(a_{m})-\chi_{Y_{m}\setminus Y_{m-1}}c\| \leq\|\Phi(a_{m})-\chi_{Y_{m}\setminus Y_{m-1}}\Phi(a)\| +\eps/2.
\]
Since $\|\chi_{Y_m}\Phi(a_m)- \Phi(a_m)\|\to 0$ and $\|\chi_{Y_{m-1}}\Phi(a_m)\|\to 0$ as $m\to\infty$, we have that $\|\Phi(a_m)-\chi_{Y_{m}\setminus Y_{m-1}}\Phi(a)\|\to 0$ as $m\to\infty$, hence $\limsup_m\|\Phi(a_{m})-\chi_{Y_{m}\setminus Y_{m-1}}c\|\leq\eps/2$. Since $\propg(\chi_{Y_{m}\setminus Y_{m-1}}c)\leq \propg(c)$, this contradicts the fact that $a_m$ cannot be $\eps$-$m$ approximated for all $m\in\N$.
\end{proof}

\section{Geometric conditions on metric spaces}\label{SectionGeomCond}

Theorem \ref{ThmRigidityRoeCorona} applies to every pair of metric spaces such that all of their sparse subspaces yield only compact ghost projections. In this section, we prove some technical results which depend on this geometric condition.

%

The conclusion of the following result should be compared with the notion of an isomorphism $\cstu(X)\to \cstu(Y)$ being \emph{rigid} (see \cite[\S1]{BragaFarah2018Trans}, or \cite[Lemma~4.6]{SpakulaWillett2013AdvMath}). 

\begin{proposition}\label{PropThankGodForNoNoncompactGhostProj}
Let $X$ be a metric space such that all of its sparse subspaces yield only compact ghost projections. Let $(p_n)_n$ be an orthogonal sequence of non-zero finite rank projections in $\cstu(X)$ such that $\sum_{n\in M}p_n$ converges in the strong operator topology to an element in $\cstu(X)$ for all $M\subseteq \N$. Then 
\[
\delta=\inf_n \sup_{x,y\in X} |\langle p_n\delta_x,\delta_y\rangle|>0.
\]
\end{proposition}

\begin{proof}
Suppose that the conclusion fails. By going to a subsequence of $(p_n)_n$, we may assume that $\sup_{x,y\in X} |\langle p_n \delta_x, \delta_y\rangle|<2^{-n}$ for all $n\in\N$. 
 
\begin{claim}
By going to a subsequence of $(p_n)_n$, there exists a sequence $(X_n)_n$ of disjoint finite subsets of $X$ and a sequence of projections $(q_n)_n$ such that 
\begin{enumerate}
\item $d(X_k,X_m)\to \infty$ as $k+m\to \infty$, 
\item $\|p_n-q_n\|<2^{-n}$, and 
\item $q_n\in \cB(\ell_2(X_n))$, for all $n\in\N$.
\end{enumerate}
\end{claim} 
 
 \begin{proof}
We construct sequences $(q_k)_k$, $(X_k)_k$ and $(n_k)_k$ by induction as follows. Since $p_0$ has finite rank, pick a projection $q_0$ with finite support 
 such that $\|p_0-q_0\|< 2^{-1}$ and set $n_0=0$. Pick a finite $X_0\subseteq X$ such that $\supp(q_0)\subseteq X_0\times X_0$. Fix $k>0$ and assume that $X_j$, $n_j$ and $q_j$ have been defined, for all $j\leq k-1$.

Let $Z=\{x\mid d(x,\bigcup_{j\leq k-1} X_j)\leq k\}$. Since $\bigcup_{j\leq k-1}X_j$ is finite, so is $Z$. Moreover $d(X\setminus Z,\bigcup_{j\leq k-1})>k$. Since $\sum_n p_n$ strongly converges to an operator in $\mathcal B(\ell_2(X))$, for all large enough $m$ we have $\|\chi_Z p_m\| <2^{-k-2}$. 
Fix such $m$. For a sufficiently large finite $X_k\subseteq X\setminus Z$ the operator 
\[
a= \chi_{X_k} p_m \chi_{X_k}
\]
 is a positive contraction and it satisfies $\|a-p_m\|<2^{-k-1}$. 
Therefore the spectrum $\Sp(a)$ of $a$ is included in $[0,2^{-k-1}]\cup [1-2^{-k-1}, 1]$ 
and the function $f\colon \Sp(a)\to \{0,1\}$ defined by $f(t)=0$ if $t<1/2$ and $f(t)=1$ if $t>1/2$ is continuous. 
By the continuous functional calculus, $q_k=f(a)$ is a projection and $\|q_k-a\|<2^{-k-1}$. 
Therefore $\|q_k-p_m\|<2^{-k}$. In addition, we have $q_k\in \cst(a)\subseteq \chi_{X_k}\cstu(X)\chi_{X_k}$, 
 and therefore $\supp(q_k)\subseteq X_k^2$ as required. 

Let $n_k=m$. This describes the recursive construction of the sequences $(q_k)_k$, $(X_k)_k$ and $(n_k)_k$, and completes the proof. 
\end{proof}

Pass to a subsequence of $(p_n)_n$ and let $(X_n)_n$ and $(q_n)_n$ be given by the claim above. Since $\|p_n-q_n\|<2^{-n}$, $\sum_n(p_n-q_n)$ is compact, which implies $\sum_n(p_n-q_n)\in\cstu(X)$. Since $\sum_np_n$ converges in the strong operator topology to an element in $\cstu(X)$, so does $\sum_nq_n$. Therefore, since $\sum_nq_n\in \bigoplus_n\cB(\ell_2(X_n))$, it follows that $\sum_nq_n\in \cstu(\tilde X)$, where $\tilde X=\bigsqcup_nX_n$. 

Let us notice that $\sum_nq_n$ is a noncompact ghost projection. Indeed, $\sum_nq_n$ is clearly a noncompact projection. Also, since $\|p_n-q_n\|<2^{-n}$, it follows that $\sup_{x,y\in X}|\langle q_n\delta_x,\delta_y\rangle|<2^{-n+1}$ for all $n\in\N$. So, $\sum_{n} q_n$ is a ghost projection (see \cite[Claim 2 in proof of Theorem 6.1]{BragaFarah2018Trans} for details); contradiction. 
\end{proof}

The next result will be essential in the proof of Proposition \ref{X:PropCofiniteSubRank1.General}.

\begin{lemma}\label{LemmaToolFiniteRank.General}
Let $X$ be a metric space and let $(p_n)_n$ be an orthogonal sequence of finite rank projections in $\cB(\ell_2(X))$ such that
\begin{enumerate}
\item Each $p_n$ has rank strictly greater than $1$, 
\item The sum $\sum_{n\in S} p_n$ converges in the strong operator topology to an element of $\cstu(X)$, for every $S\subseteq \bbN$, and 
\item\label{Item1111111} $\inf_{n\in\N}\sup_{x\in X}\|p_n\delta_x\|>0$.
\end{enumerate}
Then there exist an infinite $M\subseteq \N$ and sequences of non-zero 
projections $(q_n)_n$ and $(w_n)_n$ such that $q_n+w_n=p_n$ for all $n\in \N$, and both 
 $\sum_{n\in S}q_n$ and $\sum_{n\in S}w_n$ converge in the strong operator topology to 
 an element of $\cstu(X)$, for every $S\subseteq M$. 
 \end{lemma}

We point out that, by Proposition \ref{PropThankGodForNoNoncompactGhostProj}, \eqref{Item1111111} above is satisfied automatically if all sparse subspaces of $X$ yield only compact ghost projections. 

\begin{proof}[Proof of Lemma \ref{LemmaToolFiniteRank.General}]
By the hypothesis, there are $\delta>0$ and a sequence $(x_n)_n$ in $X$ such that $\|p_n\delta_{x_n}\|>\delta$ for all $n\in\N$. Using the \cstar-equality, we have that
\[
\|p_ne_{x_nx_n}p_n\|=\|p_ne_{x_nx_n}\|^2=\|p_n\delta_{x_n}\|^2 \geq \delta^2,
\]
for all $n\in\N$. For $n\in\N$, set 
\[
q_n=p_ne_{x_nx_n}p_n/\|p_ne_{x_nx_n}p_n\|. 
\]
Then $q_n$ is a rank 1 projection and $q_n\leq p_n$. Since $q_n$ has rank 1, $q_n<p_n$. 
Set $w_n=p_n-q_n$. We are left to show that there is an infinite $M$ such that both $\sum_{n\in S}q_n$ 
and $\sum_{n\in S} w_n$ converges in the strong operator topology to elements of $\cstu(X)$, for all $S\subseteq M$.

 \begin{claim} \label{Claim54} There exists an infinite $M\subseteq \bbN$ such that 
 $\sum_{n\in L} q_n$ strongly converges to an operator in $\cstu(X)$ for every $L\subseteq M$. 
 \end{claim} 
 
 \begin{proof} First, by going to a subsequence, assume that $\alpha=\lim_n\|p_ne_{x_nx_n}p_n\|$ exists. Notice that $\alpha\geq \delta$. Also, since all $p_n$ have finite rank, we can find an infinite $M\subseteq \bbN$ such that 
 \begin{equation}\label{Eq.aca}
 \|p_n e_{x_m x_m}\|<2^{-m-n}
 \end{equation}
 for all distinct $m,n\in M$. 
 
Fix $L\subset M$. Define $a=\alpha^{-1}\sum_{n\in L} p_n$ and 
 $c=\sum_{m\in L} e_{x_m x_m}$, so $a,c\in \cstu(X)$. Notice that the operator $aca -\sum_{n\in L} q_n$ is compact. Indeed, by \eqref{Eq.aca}, 
 this follows from the fact that 
 \[\sum_{n\in L}(\alpha^{-1} p_ne_{x_nx_n}p_n-q_n)\]
is compact. At last, since $a$ and $c$ belong to $\cstu(X)$, so does $\sum_{n\in L} q_n$. 
 \end{proof} 
 
Since $w_n=p_n-q_n$ for all $n\in\N$, Claim \ref{Claim54} implies that $\sum_{n\in L} w_n\in \cstu(X)$ for every 
 $L\subseteq M$ and the conclusion 
 follows. 
\end{proof}

\section{From almost lifts to lifts}\label{S:proof.blackbox} 
 
In this section, we prove Theorem~\ref{T:blackbox}. The proof of the following resembles the techniques used in \cite[Proposition 8.5 and Lemma 8.7]{MKV.FA}, and shows that in case the underlying map $\Lambda$ between quotients is an isomorphisms, almost lifts are indeed lifts (see Definition~\ref{def:lift} and ~\ref{almost.Def.LiftOnDiagIso}).

\begin{proposition}\label{PropCofiniteSubRank1}
Suppose $X$ and $Y$ are countable metric spaces and $\Lambda$ is an isomorphism between 
their uniform Roe coronas which is almost liftable on the diagonal. 
Then $\Lambda$ is liftable on the diagonal. 
\end{proposition}

\begin{proof} 
Let $\Phi\colon \ell_\infty(X)\to \cB(\ell_2(Y))$ be a strongly continuous $^*$-homomorphism which almost lifts $\Lambda$ on the diagonal. Fix a nonmeager ideal $\SI\subseteq \mathcal{P}(X)$ such that $\Phi$ lifts $\Lambda$ on $\{\pi_X(\chi_S)\colon S\in \SI\}$. 

We shall prove that the following conditions hold.
\begin{enumerate}
\item\label{nonm:c2} If $q\in \roeq(Y)$ commutes with $\pi_Y(\Phi(\chi_S))$ for all $S\in\SI$ then $q$ commutes with $\pi_Y(\Phi(\chi_S))$ for all $S\subseteq X$,
\item\label{nonm:c3} $\pi_Y(\Phi(1))=1$. 
\item\label{nonm:c4} $\Phi$ lifts $\Lambda$ on $\ell_\infty(X)/c_0(X)$. 
\pushcounter
\end{enumerate}
\eqref{nonm:c2} Suppose the statement fails for $q\in \roeq(Y)$ and pick $b\in \cstu(Y)$ with $q=\pi_Y(b)$. Fix $S\subseteq X$ and $\eps>0$ such that $\norm{[q,\pi_Y(\Phi(\chi_S))]}>\eps$. 
We shall find sequences $F(n)$ and $Y(n)$, for $n\in \bbN$ that satisfy the following. 
\begin{enumerate}
\popcounter
\item\label{L5.6.2} The sets $F(n)$ are finite and disjoint subsets of $S$. 
\item \label{L6.6.2} The sets $Y(n)$ are finite and disjoint subsets of $Y$. 
\item \label{L7.6.2} The following hold for all $m\neq n$. 
\begin{enumerate}
\item \label{L7a.6.2} $\|\Phi(\chi_{F(m)})\chi_{Y(n)}\|<2^{-m-n}\e$,
\item \label{L7b.6.2} 
$\|\Phi(\chi_{F(m)})b \chi_{Y(n)}\|<2^{-m-n}\e$, and 
\item \label{L7c.6.2} 
$\|\chi_{Y(n)}[b,\Phi(\chi_{F(n)})]\chi_{Y(n)}\|>\e$.
\end{enumerate}
\pushcounter
\end{enumerate}
Note that we have
 \begin{enumerate}
 \popcounter
\item \label{L8.6.2} $\Phi(\chi_S)=\sum_{x\in S}\Phi(\chi_{\{x\}})$
\pushcounter
\end{enumerate}
 where the series on the right-hand side converges in the strong topology. 
Since $\|[q,\pi_Y(\Phi(\chi_S))]\|>\e$, we can find a finite $Y(1)\subset Y$ large enough to have  $\|\chi_{Y(1)}[b,\Phi(\chi_{S})]\chi_{Y(1)}\|>\e$. Since the series in \eqref{L8.6.2}
strongly converges to $\Phi(\chi_{S})$, we can find a finite $F(1)\subset S$ large enough to satisfy \eqref{L7c.6.2} for $n=1$. 

By applying 
Lemma~\ref{L:AA} 
with $b$ being each of $\chi_{Y(1)}, b\chi_{Y(1)}$, and $\chi_{Y(1)}b$, 
we obtain a finite subset $E\subset X$ such that for all $L\subseteq X\setminus E$
all of $\Phi(\chi_L)\chi_{Y(1)}$, $\chi_{Y(1)}\Phi(\chi_L)$, 
$\Phi(\chi_L)b\chi_{Y(1)}$, and $\chi_{Y(1)}b\Phi(\chi_L)$ have norm smaller than $2^{-3}\e$. 

With $S'=S\setminus E$ we have 
$\|[b,\Phi(\chi_{S'})]\|\geq\|[q,\pi_Y(\Phi(\chi_{X'}))]\|>\e$, and we can find a finite $Y(2)\subseteq Y\setminus Y(1)$ large enough to satisfy $\|\chi_{Y(2)}[b,\Phi(\chi_{S'})]\chi_{Y(2)}\|>\e$. Since the series in \eqref{L8.6.2}
strongly converges to $\Phi(\chi_{S})$, we can find a finite $F(2)\subseteq S'$ large enough  to satisfy \eqref{L7c.6.2} for $n=2$.

Proceeding in this manner, we construct the sequences $F(n)$, and $Y(n)$ with the required
properties. 

Since $\SI$ is nonmeager and contains all finite sets, by Proposition~\ref{T:J-NT} there is
an infinite $L\subseteq \bbN$ such that $F=\bigcup_{j\in L} F(j)$ belongs to $ \SI$. Hence, by our assumptions, $q$ commutes with $\pi_Y(\Phi(\chi_{F}))$. However, \eqref{L7.6.2} implies that $[b,\Phi(\chi_{F})]$ is not compact; contradiction. 

 \eqref{nonm:c3} Suppose for a contradiction that $b\coloneqq 1-\Phi(1)$ is not compact. Lemma~\ref{prop:stayinside} implies $b\in \cstu(Y)$ and therefore 
 $q\coloneqq \pi_Y(b)$ is a nonzero projection in $\roeq(Y)$.

 Since $q$ commutes with  $\pi_Y(\Phi(a))$ for all $a\in\ell_\infty(X)$ and $\Phi$ lifts $\Lambda$ on $\SI$, 
 $\Lambda^{-1}(q)$ commutes with $\pi_X(\chi_S)$ for all $S\in\SI$. By \ref{nonm:c2}, $\Lambda^{-1}(q)$ commutes with $\chi_S$ for all $S\subseteq \NN$. The canonical copy of $\ell_\infty(X)$ is a masa \footnote{A masa is a maximal abelian self-adjoint subalgebra} in $\cB(\ell_2(X))$, therefore by Johnson-Parrott's Theorem (\cite{JohPar}, see also \cite[Theorem~12.3.2]{Fa:Combinatorial}) $\ell_\infty(X)/c_0(X)$ is a masa in the Calkin algebra, and therefore in $\roeq(X)$. This implies that $\Lambda^{-1}(q)$ is a projection in $\ell_\infty(X)/c_0(X)$.  Fix $S\subseteq \bbN$ such that $\Lambda^{-1}(q)=\pi_X(\chi_S)$.

Since $\SI$ is nonmeager, it is dense (Theorem~\ref{T:J-NT}), hence there exists an infinite $S'\subseteq S$ such that $S'\in\SI$. Since $\chi_{S'}\leq 1$, it follows that $\Phi(\chi_{S'})b$ is compact, hence $\pi_Y(\Phi(\chi_{S'}))q=0$. On the other hand, $S'\in\SI$ implies that $\pi_Y(\Phi(\chi_{S'}))=\Lambda(\pi_X(\chi_{S'}))\leq q$ and $\pi_Y(\Phi(\chi_{S'}))=0$. This is a contradiction, since $S'$ is infinite.

\eqref{nonm:c4} We will show that $\pi_Y(\Phi(\chi_S))=\Lambda(\pi_X(\chi_S))$ for 
all $S\subseteq X$. 

Consider $\cP(X)$ with respect to the compact metric topology inherited from $2^X$. 
Also consider the unit ball of $\cB(\ell_2(Y))$ with respect to the strong operator topology. 
It is a Polish space, and it has the unit ball of $\mathcal K(\ell_2(Y))$ as a Borel subset 
 (e.g., \cite[Lemma~2.5]{Farah.C}). 

Since $\Phi$ is strongly continuous, the function 
\[
\cP(X)\ni L\mapsto \Phi(\chi_L)\in \cB(\ell_2(Y))
\]
is continuous. 
Let $F\colon \cstu(X)\to \cstu(Y)$ be any lifting for $\Lambda$ and fix $S\subseteq X$. 
Since $\Phi(\chi_S)$ and $F(\chi_S)$ are fixed, 
the ideal 
\[
\SI_{S}=\{T\subseteq X\colon \Phi(\chi_T)(\Phi(\chi_S)-F(\chi_S))\in\mathcal K(\ell_2(Y))\}
\]
is, as a continuous preimage of a Borel set, Borel itself. 
For every $T\in\SI$ each of 
$\Phi(\chi_T)\Phi(\chi_S)-\Phi(\chi_{S\cap T})$, $\Phi(\chi_T)-F(\chi_T)$,
 and $ F(\chi_T)F(\chi_S)-F(\chi_{S\cap T})$ is compact.
 Therefore, $\SI\subseteq \SI_S$. 

Suppose $X\notin \SI_S$. 
Since $\SI_S$ is a proper Borel ideal that includes all finite subsets of $X$, it is meager
(see e.g., \cite[\S3.10]{FarahBook2000}). This implies that $\SI$ is meager; contradiction.

 Therefore $X\in \SI_S$ and 
\[
\Phi(1)(\Phi(\chi_S)-F(\chi_S))\in\mathcal K(\ell_2(Y)).
\]
This, together with \ref{nonm:c3}, implies that $\Phi(\chi_S)-F(\chi_S)\in\mathcal K(\ell_2(Y))$. Since $S$ was arbitrary, \eqref{nonm:c4} follows.
This concludes the proof. 
 \end{proof}

We can now prove our lifting result.

\begin{proof}[Proof of Theorem~\ref{T:blackbox}]
Assume $\OCA$ and $\MA_{\aleph_1}$, and fix an isomorphism $\Lambda$ 
between the uniform Roe coronas of countable metric spaces $X$ and $Y$. 
By Theorem~\ref{ThmAlessandroPaulHomomorphism}, 
$\Lambda$ is almost liftable on the diagonal. 
By Proposition~\ref{PropCofiniteSubRank1}, $\Lambda$ is liftable on the diagonal. 
\end{proof}

\section{From lifts to coarse equivalence}
\label{SectionCoarseEquiv}

This section is devoted to the proof of Theorem~\ref{ThmRigidityRoeCorona}. As a corollary to Theorem~ \ref{T:blackbox} and Theorem \ref{ThmRigidityRoeCorona}, we also obtain Corollary \ref{Cor.ThmRigidityRoeCorona}.
 
 \begin{proposition}\label{X:PropCofiniteSubRank1.General}
Suppose $X$ and $Y$ are countable metric spaces and $\Lambda$ is an isomorphism between 
the uniform Roe coronas of $X$ and $Y$. Moreover, let $\Phi\colon \ell_\infty(X)\to \cB(\ell_2(Y))$ be a strongly continuous $^*$-homo\-mor\-phism which almost lifts $\Lambda$ on the diagonal and such that  
\[
\inf_{x\in X'}\sup_{y\in Y}\|\Phi(e_{xx})\delta_y\|>0
\]
for some cofinite $X'\subseteq X$. Then  $\Phi(e_{xx})$ is a rank 1 projection for cofinitely many $x\in X$.
\end{proposition} 

 \begin{proof}
Let $X'\subseteq X$ be as in the statement of the lemma. In particular, $\Phi(e_{xx})\neq 0$ for all $x\in X'$. It remains to prove that $\Phi(e_{xx})$ has rank not greater than 1 for cofinitely many $x\in X'$. Suppose otherwise and fix a sequence $(x_n)_n$ in $X$ of distinct elements such that $p_n=\Phi(e_{x_nx_n})$ has rank at least 2 for all $n\in\N$. As $X'$ is cofinite in $X$, discarding some elements if necessary, we can assume that $(x_n)_n$ is a sequence in $X'$.
 Proposition~\ref{PropCofiniteSubRank1} implies that $\Phi$ lifts $\Lambda$ on $\ell_\infty(X)/c_0(X)$
and therefore $\sum_{n\in S} p_n\in \cstu(Y)$ for all $S\subseteq \bbN$. By Lemma \ref{LemmaToolFiniteRank.General} there are an infinite set $M\subseteq\N$ and nonzero projections $(q_n)_n$ with $q_n<p_n$ and $q=\sum_{n\in M} q_n\in \cstu(X)$. Since $q$ commutes with $\Phi[\ell_\infty(X)]$, $\pi_Y(q)$ commutes with $\Lambda(\ell_{\infty}(X)/c_0(X))$. Since $\ell_\infty(X)/c_0(X)$ is a masa in $\roeq(X)$, and $\Lambda$ is an isomorphism, $\Lambda(\ell_{\infty}(X)/c_0(X))$ is a masa in $\roeq(Y)$ and we have that $\pi_Y(q)\in\Lambda(\ell_{\infty}(X)/c_0(X))$. On the other hand, $\norm{\pi_Y(q-\Phi(\chi_S))}\geq 1$ for all $S\subseteq \bbN$; contradiction.
\end{proof}

 Lemma \ref{LemmaDefiMaps.General} below will allow us to define maps $f\colon X\to Y$ and $g\colon Y\to X$ which will witness that $X$ and $Y$ are coarsely equivalent. 
This lemma, as well as the auxiliary Lemma~\ref{LemmaDefiMaps0}, 
have the following mouthful of an assumption as the starting point.

\begin{assumption}\label{Assumption}
Suppose $(X,d)$ and $(Y,\partial)$ are countable metric spaces and
$\Lambda\colon \roeq(X)\to \roeq(Y)$ 
is an isomorphism which is liftable on diagonals. Let 
 $\Phi\colon \ell_\infty(X)\to \cstu(Y)$ and $\Psi\colon \ell_\infty(Y)\to \cstu(X)$ be strongly continuous $^*$-homomorphisms which lift $\Lambda$ on $\ell_\infty(X)/c_0(X)$ and $\Lambda^{-1}$ on $\ell_\infty(Y)/c_0(Y)$, respectively. 
\end{assumption}

\begin{lemma}\label{LemmaDefiMaps0} Suppose $X,Y,\Lambda,\Phi$, and $\Psi$ are as in Assumption~\ref{Assumption}. Let $X'\subseteq X$, $\epsilon>0$, and $f\colon X'\to Y$. If $\|\Phi(e_{xx})e_{f(x)f(x)}\|> \epsilon$ for all $x\in X'$, then 
\[
\{x\in X'\colon \|\Psi(e_{f(x)f(x)})e_{xx}\|<\eps-\delta\}
\]
 is finite for all $\delta>0$. On the other hand, if $\|\Psi(e_{f(x)f(x)})e_{xx}\|>\epsilon$ for all $x\in X'$, then 
 \[
\{x\in X'\colon \|\Phi(e_{xx})e_{f(x)f(x)}\|<\eps-\delta\}
\]
 is finite for all $\delta>0$.
\end{lemma}

\begin{proof}
Suppose the first statement fails for $\delta>0$ and pick a sequence $(x_n)_n$ in $X'$ such 
that $\|\Psi(e_{f(x_n)f(x_n)})e_{x_nx_n}\|<\epsilon-\delta$ for all $n$. Since $\Phi$ is a strongly continuous $^*$-homomorphism and $\|\Phi(e_{xx})e_{f(x)f(x)}\|> \epsilon$ for all $x\in X'$, by going to a subsequence, we can assume that $(f(x_n))_n$ is a sequence of distinct elements. Therefore, by Lemma~\ref{L:AA}
and passing to a further subsequence, we can assume that 
\[
\max\Big(\norm{\Phi(e_{x_nx_n})e_{f(x_m)f(x_m)}},\norm{ \Psi(e_{f(x_n)f(x_n)})e_{x_mx_m}}\Big)
<2^{-m-n-1}\epsilon
\]
 whenever $n\neq m$. Let 
\[
q=\sum_n e_{f(x_n)f(x_n)}\text{ and }p=\sum_n \Phi(e_{x_nx_n}).
\]
Then $\norm{\pi_Y(pq)}\geq\epsilon$ by hypothesis, but
\begin{align*}
\norm{\Lambda^{-1}(\pi_Y(pq))}&=\norm{\Lambda^{-1}(\pi_Y(\Phi\Big(\sum_n e_{x_nx_n}\Big)))\Lambda^{-1}(\pi_Y(q))}\\
&= \norm{\pi_X\Big(\big(\sum_ne_{x_nx_n}\big)\Psi(q)\Big)} \\
&=\limsup_n\norm{\Psi(e_{f(x_n)f(x_n)})e_{x_nx_n}}\\
&<\epsilon-\delta,
\end{align*}
a contradiction.

The second statement follows analogously.
\end{proof}

\begin{lemma}\label{LemmaDefiMaps.General}
Suppose $X,Y,\Lambda,\Phi$, and $\Psi$ are as in Assumption~\ref{Assumption}. Suppose in addition that \[\inf_{x\in X'}\sup_{y\in Y}\|\Phi(e_{xx})\delta_y\|>0\]
for some cofinite $X'\subseteq X$. 
Then there exist $\delta>0$, a cofinite $X''\subseteq X'$, and a map $f\colon X''\to Y$ such that 
\[
\|\Phi(e_{xx})e_{f(x)f(x)}\|\geq \delta\text{ and }\|\Psi(e_{f(x)f(x)})e_{xx}\|\geq \delta,
\]
for all $x \in X''$. 
\end{lemma}

\begin{proof}
Let $X'$ be as in the statement of the lemma, and pick $\delta>0$ and $f\colon X'\to Y$ such that $\|\Phi(e_{xx})e_{f(x)f(x)}\|\geq \delta$ for all $x\in X'$. By Lemma \ref{LemmaDefiMaps0}, replacing $\delta$ by a smaller $\delta$ if necessary, we can pick a cofinite $X''\subseteq X'$ such that $f\restriction X''$ has the required property.
\end{proof}

Our next goal is to prove that the maps given by Lemma \ref{LemmaDefiMaps.General} are coarse. First, we need a simple lemma.

\begin{lemma}\label{PropNormEqualityRank1}
Let $H$ be a Hilbert space, $p,q,a,b,c\in \cB(H)$ and $\beta\in \mathbb R$. Assume that $p$, $q$, $a$ and $\beta b$ are projections, $a$ has rank 1, $c=ac$, and $c^*c=b$. Then 
\[
\|pcq\|=\sqrt{\beta}\cdot \|ap\|\cdot\|bq\|.
\]
In particular, if $a=b=c$, it follows that $\|paq\|=\|ap\|\cdot\|aq\|$.
\end{lemma}

\begin{proof}
The \cstar-equality gives that $\|pcq\|^2=\|pcqc^*p\|$. Since $a$ is a rank $1$ projection and $cqc^*=acqc^*a$, there exists a positive $\lambda \in \mathbb R$ such that $\lambda a=cqc^*$. Therefore, using the \cstar-equality, it follows that \[\|pcq\|^2=\lambda\cdot \|pap\|=\lambda\cdot\|ap\|^2.\]
Using \cstar-equality unsparingly, we can compute $\lambda$. Indeed, 
\[\lambda=\|cqqc^*\|=\|cq\|^2=\| qc^*cq\|=\|qbq\|=\beta\cdot \|bq\|^2.\]
This finishes the proof.
\end{proof}

\begin{lemma}\label{LemmaTheMapsAreCoarse}
Suppose $X,Y,\Lambda,\Phi$, and $\Psi$ are as in Assumption~\ref{Assumption}. Suppose in addition that \[\inf_{x\in X'}\sup_{y\in Y}\|\Phi(e_{xx})\delta_y\|>0\]
for some cofinite $X'\subseteq X$. 
 Then for all $r,\delta>0$ 
there exists $s>0$ such that for all $x_1,x_2\in X$ and all $y_1,y_2\in Y$ the following holds:

If $d(x_1,x_2)\leq r$, 
$\|\Phi(e_{x_1x_1})e_{y_1y_1}\|\geq \delta$, and $\|\Phi(e_{x_2x_2})e_{y_2y_2}\|\geq \delta$,
then $\partial (y_1,y_2)\leq s$.
\end{lemma}

\begin{proof}
Suppose otherwise. Then there exist 
 $r,\delta>0$, 
 sequences $(x^1_n)_n$ and $(x^2_n)_n$ in $X$, and sequences $(y^1_n)_n$ and $(y^2_n)_n$ in $Y$ such that $d(x_n^1,x_n^2)\leq r$, 
$\|\Phi(e_{x_n^1x_n^1})e_{y_n^1y_n^1}\|\geq \delta$, $\|\Phi(e_{x_n^2x_n^2})e_{y_n^2y_n^2}\|\geq \delta$, and 
$\partial(y_n^1,y_n^2)\geq n$ for all $n\in\N$. Let $X_1=\{x_n^1\}_n$, $X_2=\{x_n^2\}_n$, $Y_1=\{y_n^1\}_n$, and $Y_2=\{y_n^2\}_n$.

\begin{claim}
The sets $X_1,X_2,Y_1$ and $Y_2$ are infinite.
\end{claim}
\begin{proof}
Since $\{\partial(y_1,y_2)\colon y_1\in Y_1, y_2\in Y_2\}$ is unbounded, $Y_1\cup Y_2$ is infinite. Hence $Y_1$ or $Y_2$ must be infinite. 

Since $X$ is locally finite, $X_2\subseteq\{x\in X\colon\exists x_1\in X_1,\ d(x,x_1)\leq r\}$, 
 and $X_1\subseteq\{x\in X\colon\exists x_2\in X_2, d(x,x_2)\leq r\}$, we conclude 
that $X_1$ is finite if and only if $X_2$ is finite. 

Also, if $X_1$ is finite, then so is $Y_1$. Indeed, as $\Phi(e_{xx})$ is compact for all $x\in X$, the set $\{y\in Y\colon \|\Phi(e_{xx})e_{yy}\|\geq \delta\}$ is finite for all $x\in X$. Moreover, since $\sum_{x\in X}\Phi(e_{xx})$ converges in the strong operator topology, the set $\{x\in X\colon \|\Phi(e_{xx})e_{yy}\|\geq \delta\}$ is finite for every $y\in Y$. In particular, $X_1$ is finite if and only if $Y_1$ is finite. Similarly, $X_2$ is finite if and only if $Y_2$ is finite. This finishes the proof.
\end{proof}

 By Proposition~\ref{X:PropCofiniteSubRank1.General},
 we can assume that $\Phi(e_{x_n^1x_n^1})$ and $\Phi(e_{x_n^2x_n^2})$ have rank~1 for all $n$. By Lemma \ref{L:AA}, and going to a subsequence if necessary, we can find a sequence of finite disjoint subsets $(Z_n)_n$ of $Y$, and sequences of rank 1 projections $(a_n)_n$ and $(b_n)_n$ such that $\|a_n-\Phi(e_{x_n^1,x_n^1})\|< 2^{-n}$, $\|b_n-\Phi(e_{x_n^2,x_n^2})\|<2^{-n}$, and $a_n,b_n\in \cB(\ell^2(Z_n))$ for all $n\in\N$. By going to a subsequence, assume that $\|a_ne_{y_n^2y_n^2}\|\geq \delta/2$ and $\|b_ne_{y_n^1y_n^1}\|\geq \delta/2$ for all $n\in\N$.

Let $s=\sum_ne_{x_n^2x_n^1}$. Then $s\in \cB(\ell_2(X))$ and
$\propg(s)=\sup_n d(x_n^1,x_n^2) \leq r$. Pick $e\in\cstu(Y)$ which satisfies $\pi_Y(e)=\Lambda(\pi_X(s))$.

\begin{claim}
$\lim_n\|a_neb_n-\chi_{Z_n}e\chi_{Z_n}\|=0$.
\end{claim}

\begin{proof}
Since $s=s(\sum_ne_{x_n^1x_n^1})$ and $s=(\sum_ne_{x_n^2x_n^2})s$, we have that both $(\sum_n\Phi(e_{x_n^2x_n^2}))e-e$ and $e(\sum_n\Phi(e_{x_n^1x_n^1}))-e$ are compact. Hence, since $\sum_n(a_n-\Phi(e_{x_n^2x_n^2}))$ and $\sum_n(b_n-\Phi(e_{x_n^1x_n^1}))$ are compact, it follows that $(\sum_na_n)e-e$ and $e(\sum_nb_n)-e$ are compact. Therefore, it follows that $\lim_n\|\chi_{Z_n}eb_n-\chi_{Z_n}e\chi_{Z_n}\|=0$ and $\lim_n\|a_ne\chi_{Z_n}-\chi_{Z_n}e\chi_{Z_n}\|=0$. 
Since $\chi_{Z_n}b_n=b_n$ and $a_n\chi_{Z_n}=a_n$, the claim follows.
\end{proof}

\begin{claim}\label{Claim.lim0}
$\lim_n\|\chi_{\{y^2_n\}}a_neb_n\chi_{\{y^1_n\}}\|=0$.
\end{claim}

\begin{proof}
First notice that, since $e\in \cstu(Y)$ and $\lim_n\partial (y^1_n,y^2_n)=\infty$, it follows that $\lim_n\|\chi_{\{y^2_n\}}e\chi_{\{y^1_n\}}\|=0$. In particular, 
\[\lim_n\|\chi_{\{y^2_n\}}\chi_{Z_n}e\chi_{Z_n}\chi_{\{y^1_n\}}\|\leq \lim_n\|\chi_{\{y^2_n\}}e\chi_{\{y^1_n\}}\|=0.\]
Therefore, by Claim \ref{Claim.lim0}, $\lim_n\|\chi_{\{y^2_n\}}a_neb_n\chi_{\{y^1_n\}}\|=0$.
\end{proof}

\begin{claim}
$\lim\inf_n\|a_neb_n\|>0$.
\end{claim}

\begin{proof}
If the claim fails, pick an infinite $M\subset \N$ such that $\sum_n\|a_neb_n\|<\infty$. So, $\sum_{n\in M}a_neb_n$ is compact. Going to a further infinite $M\subset \N$ if necessary, assume also that $\|a_neb_m\|<2^{-n-m}$ for all $n,m\in M$ with $n\neq m$. So, $\sum_{n\in M}a_neb_n-(\sum_{n\in M}a_n)e(\sum_{n\in M}b_n)$ is compact. 

Therefore, as 
\[
\textstyle (\sum_{n\in M}e_{x^2_nx^2_n})s(\sum_{n\in M}e_{x^1_nx^1_n})=\sum_{n\in M}e_{x^2_nx^1_n}
\]
is not compact, and as $\sum_{n\in M}(a_n-\Phi(e_{x_n^2x_n^2}))$ and $\sum_{n\in M}(b_n-\Phi(e_{x_n^1x_n^1}))$ are compact, this shows that 
\begin{align*}
\pi_Y\Big(\sum_{n\in M}a_neb_n\Big)&=\pi_Y\Big(\sum_{n\in M}a_n\Big)\pi_Y(e)\pi_Y\Big(\sum_{n\in M}b_n\Big)\\
&=\pi_Y\Big(\sum_{n\in M}\Phi(e_{x^2_nx^2_n})\Big)\pi_Y(e)\pi_Y\Big(\sum_{n\in M}\Phi(e_{x^1_nx^1_n})\Big)\\
&=\Lambda\circ\pi_X\Big(\Big(\sum_{n\in M}e_{x^2_nx^2_n}\Big)s\Big(\sum_{n\in M}e_{x^1_nx^1_n}\Big)\Big)
\end{align*}
is nonzero; contradiction.
\end{proof}

Let $\gamma=\inf_n\|a_neb_n\|$, so $\gamma>0$. Pick $n\in\N$ large enough to satisfy $\|\chi_{\{y^2_n\}}a_neb_n\chi_{\{y^1_n\}}\|< \gamma \delta^2$. Since $b_n$ is a rank 1 projection, $b_ne^*a_neb_n=\beta b_n$ for some positive $\beta\in \mathbb R$. By choice of $\gamma$, it follows that $\beta=\|a_neb_n\|^2\geq \gamma^2$. Applying Lemma \ref{PropNormEqualityRank1} with $c=a_neb_n$, $a=a_n$, $b=\beta b_n$, $p=\chi_{\{y^2_n\}}$, and $q=\chi_{\{y^1_n\}}$, we have that
\[
\|\chi_{\{y^2_n\}}a_neb_n\chi_{\{y^1_n\}}\|=\sqrt{\beta}\cdot \|a_n\chi_{\{y^2_n\}}\|\cdot\|b_n\chi_{\{y^1_n\}}\|\geq\frac{ \gamma\delta^2}{4};
\]
contradiction.
\end{proof}

We now show that a more technical version of Theorem \ref{ThmRigidityRoeCorona} holds. This is useful since it does not impose any geometric restrictions on either $X$ or $Y$.

\begin{theorem}\label{ThmArtificial}
Let $X$ and $Y$ be u.l.f. metric spaces and $\Lambda:\roeq(X)\to \roeq(Y)$ be an isomorphism. Moreover, suppose there are strongly continuous $*$-homomorphisms $\Phi:\ell_\infty(X)\to \cstu(Y)$ and $\Phi:\ell_\infty(Y)\to \cstu(X)$ which lift $\Lambda$ on $\ell_\infty(X)/c_0(X)$ and $\ell_\infty(Y)/c_0(Y)$, respectively, and such that \[\inf_{x\in X'}\sup_{y\in Y}\|\Phi(e_{xx})\delta_y\|>0\text{ and }\inf_{y\in Y'}\sup_{x\in X}\|\Psi(e_{yy})\delta_x\|>0\]
for some cofinite subsets $X'\subseteq X$ and $Y'\subseteq Y$. Then $X$ and $Y$ are coarsely equivalent.
\end{theorem}

\begin{proof} 
 By Lemma \ref{LemmaDefiMaps.General}, there exist $\delta>0$, a cofinite $X'\subseteq X$, a cofinite $Y'\subseteq Y$, $f\colon X'\to Y$ and $g\colon Y'\to X$ such that 
\begin{enumerate}
\item\label{ItemfCoarse} $\|\Phi(e_{xx})e_{f(x)f(x)}\|\geq \delta$, for all $x\in X'$,
\item $\|\Psi(e_{f(x)f(x)})e_{xx}\|\geq \delta$, for all $x\in X'$,
\item\label{ItemgCoarse} $\|\Psi(e_{yy})e_{g(y)g(y)}\|\geq \delta$, for all $y\in Y'$, and
\item\label{ItemfgCloseIdentity} $\|\Phi(e_{g(y)g(y)})e_{yy}\|\geq \delta$, for all $y\in Y'$.
\end{enumerate}
Moreover, since $\Phi$ and $\Psi$ are strongly continuous, there exist cofinite subsets $X''\subseteq X'$ and $Y''\subseteq Y'$ such that $f(X'')\subseteq Y'$ and $g(Y'')\subseteq X'$.

Since $f$ and $g$ satisfy \ref{ItemfCoarse} and \ref{ItemgCoarse} above, respectively, Lemma \ref{LemmaTheMapsAreCoarse} implies that $f\restriction X'$ and $g\restriction Y'$ are coarse. Since $X'$ and $Y'$ are cofinite and since $X$ and $Y$ are locally finite, this shows that $f$ and $g$ are coarse. 

Let us verify that $g\circ f$ and $f\circ g$ are close to $\Id_X$ and $\Id_Y$, respectively. Since $f$ and $g$ satisfy \ref{ItemfCoarse} and \ref{ItemfgCloseIdentity} above and since $g(Y'')\subseteq Y'$, it follows that \[\|\Phi(e_{g(y)g(y)})e_{f(g(y))f(g(y))}\|\geq \delta\text{ and }\|\Phi(e_{g(y)g(y)})e_{yy}\|\geq \delta,\] for all $y\in Y''$. So, Lemma \ref{LemmaTheMapsAreCoarse} implies that 
\[\sup\{\partial (y,f\circ g(y))\colon y\in Y''\}<\infty.\]
Since $Y\setminus Y''$ is finite, this shows that $f\circ g$ is close to the identity $\Id_Y$. Similar arguments show that $g\circ f$ is close to $\Id_X$. 
\end{proof}

\begin{proof}[Proof of Theorem \ref{ThmRigidityRoeCorona}]
Let $\Lambda:\roeq(X)\to \roeq(Y)$ be an isomorphism which is liftable on the diagonals. We only need to notice that the hypothesis of Theorem \ref{ThmArtificial} are satisfied for $\Lambda:\roeq(X)\to \roeq(Y)$. As $\Lambda$ is liftable on the diagonals, there are strongly continuous $*$-homomorphisms $\Phi:\ell_\infty(X)\to \cstu(Y)$ and $\Psi:\ell_\infty(Y)\to \cstu(X)$ which lift $\Lambda$ on $\ell_\infty(X)/c_0(X)$ and $\ell_\infty(Y)/c_0(Y)$. For the last condition, notice that, by the lifting property of $\Phi$, we must have that $(\Phi(e_{xx}))_{x\in X}$ are non-zero finite rank projections for cofinitely many $x\in X$. Therefore, as all sparse subspaces of $X$ yield only compact ghost projections, Proposition \ref{PropThankGodForNoNoncompactGhostProj} implies that 
\[\inf_{x\in X'}\sup_{y\in Y}\|\Phi(e_{xx})\delta_y\|>0\] for some cofinite $X'\subseteq X$. The argument for $\Psi$ is exactly the same, so we are done by Theorem \ref{ThmArtificial}.
\end{proof}

\section{Bijective coarse equivalence}\label{SectionBijCoarseEquiv}

In this section we assume that our metric spaces satisfy a stronger geometric condition known as the operator norm localization property, ONL. If $X$ and $Y$ satisfy ONL, then 
the existence of an isomorphism between $\roeq(X)$ and $\roeq(Y)$ which 
is liftable on diagonals implies that $X$ and $Y$ are cofinitely bijectively coarsely equivalent. In the class of uniformly locally finite metric spaces ONL is equivalent to G. Yu's property~A (\cite[Theorem~4.1]{Sako2014}).

\begin{definition}\label{def:ONL}
A metric space $X$ has the \emph{operator norm localization property} (\emph{ONL}) if for all $\eps\in (0,1)$ and all $s>0$ there exists $r>0$ such that for every operator $a\in \cstu(X)$ with $\propg(a)\leq s$ there exists a unit vector $\xi\in \ell_2(X)$ with $\diam \{x\in X\colon \langle\xi,\delta_x\rangle\neq 0\}\leq r$ such that $\|a\xi\|\geq (1-\eps)\|a\|$.
\end{definition}

 The next lemma is the only instance in this section where the fact that the metric spaces have ONL is actively used.

\begin{lemma}\label{LemmaUdelta}
Suppose $X$ and $Y$ are metric spaces, and $Y$ has the ONL. Let $\Phi\colon \ell_\infty(X)\to \cstu(Y)$ be a coarse-like map. For all $\gamma>0$ and all $\eps>0$ there exists $r>0$ such that 
for all contractions $a\in \ell_\infty(X)$ and all $B\subseteq Y$ with $\|\chi_B\Phi(a)\|>\gamma$, there exists $D\subseteq Y$ with $\diam(D)<r$ such that \[\|\chi_{B\cap D}\Phi(a)\|\geq (1-\eps)\|\chi_{B}\Phi(a)\|.\]
\end{lemma}

\begin{proof}
Fix $\gamma>0$ and $\eps>0$, and let $\delta=\eps\gamma/3$. Since $\Phi$ is coarse-like, there exists $s>0$ such that $\Phi(a)$ is $\delta$-$s$-approximated for all contractions $a\in \ell_\infty(X)$. In particular, $\chi_B\Phi(a)$ is $\delta$-$s$-approximated for all $B\subseteq Y$. Fix a contraction $a\in \ell_\infty(X)$ and $B\subseteq Y$ with $\|\chi_B\Phi(a)\|>\gamma$. Then, there exists $b\in\cstu(Y)$ with $\propg(b)\leq s$ such that 
\[\|b-\chi_B\Phi(a)\|\leq \frac{\delta}{\gamma}\|\chi_B\Phi(a)\|.\]
Since $Y$ has ONL, there exist $r>0$ and a unit vector $\xi\in \ell_2(Y)$ (depending only on $\delta$ and $s$, i.e., on $\gamma$ and $\eps$) such that $\diam\{y\in Y\colon \xi(y)\neq 0\}\leq r$ and $\|b^*\xi\|\geq (1-\delta/\gamma)\|b^*\|$. Therefore, letting $D=\diam\{y\in Y\colon \xi(y)\neq 0\}$, it follows that $\|\chi_Db\|\geq(1-\delta/\gamma)\|b\|$. 

We conclude that 
\begin{align*}
\|\chi_{B\cap D}\Phi(a)\|&\geq \|\chi_Db\|-\|\chi_Db-\chi_D\chi_B\Phi(a)\|\\
&\geq \Big(1-\frac{\delta}{\gamma}\Big)\|b\|-\|b-\chi_B\Phi(a)\|\\
&\geq \Big(1-\frac{\delta}{\gamma}\Big)\|\chi_B\Phi(a)\|-2\|b-\chi_B\Phi(a)\|\\
&\geq \Big(1-\frac{3\delta}{\gamma}\Big)\|\chi_B\Phi(a)\|.
\end{align*}
By our choice of $\delta$, we are done.
\end{proof}

\subsection{Finding a bijective coarse equivalence}\label{SubsectionProofofTheoremPropertyA}
This section has been inspired by methods developed in \cite[\S4]{WhiteWillett2017}. 
In it we prove Theorem \ref{ThmRigidityRoeCoronaPropertyA}, that if $X$ and $Y$ are uniformly locally finite metric spaces, $X$ has ONL, and their uniform Roe coronas are liftable on diagonals isomorphic, then $X$ and $Y$ have cofinite subspaces that are bijectively coarsely equivalent. 
 
\begin{assumption} \label{Assumption.bijective} 
Throughout this subsection we fix the following objects. 
\begin{itemize}
 \item Two uniformly locally finite metric spaces $(X,d)$ and $(Y,\partial)$ with the ONL and 
 an isomorphism $\Lambda\colon\roeq(X)\to\roeq(Y)$. 
 \item Strongly continuous $^*$-homomorphisms
$\Phi:\ell_\infty(X)\to \cstu(Y)$ and $\Psi:\ell_\infty(Y)\to \cstu(X)$ which witness that $\Lambda$ is liftable on diagonals. 
\item Cofinite subsets $X'\subseteq X$ and $Y'\subseteq Y$ such that $\Phi(e_{xx})$ and $\Psi(e_{yy})$ are rank 1 projections for all $x\in X'$ and $y\in Y'$, respectively. 
\end{itemize}
\end{assumption}

Also, throughout this section, given $y\in Y$, $x\in X$, $A\subseteq X$, $B\subseteq Y$, 
and $\delta>0$, we let 
\begin{align*}
X_{y,\delta}&=\{x\in X\colon \|\Psi(e_{yy})e_{xx}\|\geq \delta\}, & X_{B,\delta}&=\cup_{y\in B}X_{y,\delta}, \\
Y_{x,\delta}&=\{y\in Y\colon \|\Phi(e_{xx})e_{yy}\|\geq \delta\}, & Y_{A,\delta}&=\cup_{x\in A}Y_{x,\delta}. 
\end{align*}
 
 Before proving Theorem \ref{ThmRigidityRoeCoronaPropertyA}, we need several lemmas regarding $X_{y,\delta}$ and $Y_{x,\delta}$.

\begin{lemma}\label{LemmaONL}
For all $\eps>0$ there exists $\delta>0$ such that 
\begin{enumerate}
\item\label{ItemLemmaONL} $\|\Phi(e_{xx})\chi_{Y_{x,\delta}}\|\geq 1-\eps$, for all $x\in X'$, and
\item $\|\Psi(e_{yy})\chi_{X_{y,\delta}}\|\geq 1-\eps$, for all $y\in Y'$.
\end{enumerate}
\end{lemma}

\begin{proof}
We only prove \ref{ItemLemmaONL}. Fix $\eps>0$. Since $Y$ has ONL, by Lemma \ref{LemmaUdelta}, there exists $r>0$ such that for all $x\in X'$ there exists $C_x\subseteq Y$ with $\diam(C_x)<r$ such that $\|\Phi(e_{xx})\chi_{C_x}\|^2\geq 1-\eps/2$. For each $x\in X'$, pick a unit vector $\xi_x\in \ell_2(X)$ such that $\Phi(e_{xx})=\langle \cdot,\xi_x\rangle$. It follows that 
\[\|\Phi(e_{xx})\chi_{A}\|=\|\chi_{A}\Phi(e_{xx})\|=\|\chi_{A}\xi_x\|,\]
 for all $A\subseteq Y$. In particular, $\|\chi_{C_x}\xi_x\|^2\geq 1-\eps/2$, for all $x\in X'$. 
 
 Let $N=\sup_{y\in Y}|B_r(y)|$ and pick a positive $\delta$ smaller than $\sqrt{\eps/(2N)}$. By the definition of $Y_{x,\delta}$, for all $x\in X'$ and all $y\in Y\setminus Y_{x,\delta}$, it follows that $|\xi_x(y)|=\|\Phi(e_{xx})e_{yy}\|<\delta$. Therefore, since $|C_x|\leq N$, it follows that 
 \begin{align*}
 \|\Phi(e_{xx})\chi_{Y_{x,\delta}}\|^2&=\sum_{y\in Y_{x,\delta}}|\xi_x(y)|^2\\
 & \geq \sum_{y\in C_x\cap Y_{x,\delta}}|\xi_x(y)|^2\\
 &=\sum_{y\in C_x}|\xi_x(y)|^2-\sum_{y\in C_x\setminus Y_{x,\delta}}|\xi_x(y)|^2\\
 &\geq 1-\frac{\eps}{2}-\delta^2 N\geq 1-\eps,
 \end{align*}
 for all $x\in X'$. 
\end{proof}

Before Lemma \ref{LemmaCardinalityXBDelta}, we need a general lemma regarding the interaction between the liftings $\Phi\colon \ell_\infty(X)\to \cstu(X)$ and $\Psi\colon \ell_\infty(Y)\to \cstu(Y)$.

\begin{lemma}\label{LemmaLimInf}
Let $(a_n)_n$ and $(b_n)_n$ be bounded sequences in $\ell_\infty(X)$ and $\ell_\infty(Y)$, respectively. Furthermore, assume that both $(a_n)_n$ and $(b_n)_n$ are sequences with disjoint finite supports. Then 
\[\liminf_n\Big|\|a_n\Psi(b_n)\|-\|\Phi(a_n)b_n\|\Big|=0.\] 
\end{lemma}

\begin{proof}
Since both $(a_n)_n$ and $(b_n)_n$ are bounded sequences with disjoint supports, $\sum_{n\in M}a_n$ and $\sum_{n\in M}b_n$ converges in the strong operator topology, for all $M\subset \N$. Since $\Phi$ and $\Psi$ are strongly continuous $^*$-homomorphisms, and since both $(a_n)_n$ and $(b_n)_n$ are sequences with finite disjoint supports, Lemma \ref{L:AA} implies that, by going to a subsequence, we can assume that $\|a_n\Psi(b_m)\|\leq 2^{-n-m}$ and $\|\Phi(a_n)b_m\|\leq 2^{-n-m}$ for all $n\neq m$, \[\Big\|\pi_X\Big(\sum_na_n\Psi(b_n)\Big)\Big\|=\limsup_n\|a_n\Psi(b_n)\|\]
 and \[\Big\|\pi_X\Big(\sum_n\Phi(a_n)b_n\Big)\Big\|=\limsup_n\|\Phi(a_n)b_n\|.\] 
In particular, \[\pi_X\Big(\Big(\sum_{n\in M}a_n\Big)\Big(\sum_{n\in M}\Psi(b_n)\Big)\Big)=\pi_X\Big(\sum_{n\in M}a_n\Psi(b_n)\Big)\] and \[\pi_Y\Big(\Big(\sum_{n\in M}\Phi(a_n)\Big)\Big(\sum_{n\in M}b_n\Big)\Big)=\pi_Y\Big(\sum_{n\in M}\Phi(a_n)b_n\Big).\] By going to a further subsequence, we can assume that $\lim_n\|a_n\Psi(b_n)\|$ and $\lim_n\|\Phi(a_n)b_n\|$ exist. Therefore, it follows that 
\begin{align*}
\Big\|\pi_X\Big(\sum_na_n\Psi(b_n)\Big)\Big\|&=\Big\|\Lambda\circ\pi_X\Big(\sum_na_n\Psi(b_n)\Big)\Big\|\\
&=\Big\|\Lambda\circ\pi_X\Big(\sum_na_n\Big)\Lambda\circ\pi_X\Big(\sum_n\Psi(b_n)\Big)\Big\|\\
&=\Big\|\pi_Y\Big(\sum_n\Phi(a_n)\Big)\pi_Y\Big(\sum_nb_n\Big)\Big\|\\
&=\Big\|\pi_Y\Big(\sum_n\Phi(a_n)b_n\Big)\Big\|,
\end{align*}
and we are done.
\end{proof}

\begin{lemma}\label{LemmaCardinalityXBDelta}
For all $\eps>0$ there exists $\delta>0$ such that 
\[\|\Psi(\chi_B)(1-\chi_{X_{B,\delta}})\|<\eps\text{ and }\|\Phi(\chi_A)(1-\chi_{Y_{A,\delta}})\|<\eps,\]
for all finite subsets $B\subseteq Y'$ and $A\subseteq X'$. In particular, if $\eps\in (0,1)$, then $|B|\leq |X_{B,\delta}|$ and $|A|\leq |Y_{A,\delta}|$ for all $B\subseteq Y'$ and $A\subseteq X'$. 
\end{lemma}

\begin{proof}We first show the statement for $\Psi$ as above. Suppose it fails. Then there exists $\eps\in (0, 1/2)$ and a sequence $(B_n)_n$ of finite subsets of $Y'$ such that $\|\Psi(\chi_{B_n})(1-\chi_{X_{B_n,1/n}})\|\geq 2\eps$, for all $n\in\N$.

\begin{claim}\label{C:7.7}
For every $n\in\N$ and every finite $F\subseteq Y $ there exists a finite $B\subseteq Y\setminus F$ such that $\|\Psi(\chi_{B} )(1-\chi_{X_{B,1/n}})\|> \eps$.
\end{claim}

\begin{proof}
If not, then there are $n\in\N$ and a finite $F\subseteq Y$ such that 
for every finite $B\subseteq Y\setminus F$ we have $\|\Psi(\chi_{B} )(1-\chi_{X_{B,1/n}})\|\leq \eps$. This implies $\|\Psi(\chi_{B} )(1-\chi_{X_{B,1/m}})\|\leq \eps$ for all finite $B\subseteq Y\setminus F$ and all $m>n$. Since $F$ is finite, pick $m>n$ large enough to satisfy  $\|\Psi(\chi_{D} )(1-\chi_{X_{D,1/m}})\|< \eps$ for all $D\subseteq F$. This implies that 
\begin{align*}
\|\Psi(\chi_{B_m})(1-\chi_{X_{B_m,1/m}})\|&\leq \|\Psi(\chi_{B_m\cap F})(1-\chi_{X_{B_m,1/m}})\|\\
&\ \ \ \ +\|\Psi(\chi_{B_m\setminus F})(1-\chi_{X_{B_m,1/m}})\|\\
&\leq \|\Psi(\chi_{B_m\cap F})(1-\chi_{X_{B_m\cap F,1/m}})\|\\
&\ \ \ \ +\|\Psi(\chi_{B_m\setminus F})(1-\chi_{X_{B_m\setminus F,1/m}})\|\\
&< 2\eps;
\end{align*}
contradiction.
\end{proof}

By Claim~\ref{C:7.7}, redefining the sequence $(B_n)_n$, we assume that $(B_n)_n$ is a sequence of disjoint finite subsets of $Y'$ such that $\|\Psi(\chi_{B_n})(1-\chi_{X_{B_n,1/n}})\|> \eps$, for all $n\in\N$. For each $n\in\N$, let $A_n=X\setminus X_{B_n,\delta}$. 

Since $(B_n)_n$ is a disjoint sequence of finite subsets and $\|\Psi(\chi_{B_n})\chi_{A_n}\|>\eps$ for all $n$, Lemma \ref{L:AA} allows us to pick a sequence $(X_n)_n$ of disjoint finite subsets of $X$ such that $\|\Psi(\chi_{B_n})\chi_{A_n\cap	 X_n}\|>\eps/2$ for all $n\in\N$. By Lemma \ref{LemmaLimInf}, it follows that
\[\liminf_n\Big|\|\Psi(\chi_{B_n})\chi_{A_n\cap X_n}\|-\|\chi_{B_n}\Phi(\chi_{A_n\cap X_n})\|\Big|=0.\]
Therefore, by going to a subsequence, we assume that 
\[
\inf_n\|\chi_{B_n}\Phi(\chi_{A_n\cap X_n})\|\geq \eps/4.
\]
 By Lemma~\ref{LemmaUdelta}, there exists $r>0$ and a sequence $(D_n)_n$ of subsets of $Y$ such that $\diam(D_n)<r$ and 
\begin{equation*}
\|\chi_{B_n\cap D_n}\Phi(\chi_{A_n\cap X_n})\|\geq (1-\eps)\|\chi_{B_n}\Phi(\chi_{A_n\cap X_n})\|,
\end{equation*}
for all $n\in\N$. Using $\|\Psi(\chi_{B_n})\chi_{A_n\cap X_n}\|>\eps/2$ for all $n\in\N$ and applying Lemma~\ref{LemmaLimInf} once again, by going to a further subsequence, we can assume that 
\begin{equation}\label{EqONL1}
\|\Psi(\chi_{B_n\cap D_n})\chi_{A_n\cap X_n}\|\geq (1-2\eps)\|\Psi(\chi_{B_n})\chi_{A_n\cap X_n}\|,\tag{$*$}
\end{equation}
for all $n\in\N$.

Since $Y$ is uniformly locally finite and $\sup_n\diam (D_n)\leq r$, there exists $N\in\N$ such that $\sup_n|D_n|<N$. Pick $\theta>0$ small enough to have  $4N\theta^{1/2} <\eps(1-2\eps) $.
By Lemma \ref{LemmaONL}, pick $n\in \N$ large enough to satisfy $\|\Psi(e_{yy})\chi_{X_{y,1/n}}\|\geq 1-\theta$ for all $y\in Y'$. Since, for all $y\in Y'$, $\Psi(e_{yy})$ is a rank 1 projection, then $\Psi(e_{yy})\chi_{X_{y,1/n}}\Psi(e_{yy})=\lambda\Psi(e_{yy})$ for some $\lambda\geq (1-\theta)^2$, and therefore $\|\Psi(e_{yy})(1-\chi_{X_{y,1/n}})\|<2\theta^{1/2}$ for all $y\in Y'$. It follows that 

\begin{align*}\label{EqONL3}
\|\Psi(\chi_{B_n\cap D_n })\chi_{A_n\cap X_n}\|&\leq\|\Psi(\chi_{B_n\cap D_n })\chi_{A_n}\|\tag{$**$}\\ 
&\leq \sum_{y\in B_n\cap D_n}\|\Psi(e_{yy})(1-\chi_{X_{y,1/n}})\|\\
&\leq 4\theta^{1/2}|D_n|\\
&\leq \frac{\eps(1-2\eps)}{2},
\end{align*}
for all $n\in\N$. Therefore, inequalities \eqref{EqONL1} and \eqref{EqONL3} imply that
 \[\|\Psi(\chi_{B_n})\chi_{A_n\cap X_n}\|\leq \frac{\eps}{2}\] for all $n\in\N$; contradiction.

We are left to show that, if $\eps\in (0,1)$ then $|B|\leq |X_{B,\delta}|$ for all $B\subseteq Y'$. Fix $\delta>0$ given by the first statement of the lemma for any $\eps\in(0,1)$. Notice that $|B|=\rank \Phi(\chi_B)$ and $|X_{B,\delta}|=\rank \chi_{X_{B,\delta}}$. Suppose $\rank \Phi(\chi_B)>\rank \chi_{X_{B,\delta}}$. Then, since $\mathrm{corank}(1-\chi_{X_{B,\delta}})=\rank \chi_{X_{B,\delta}} $, the images of the projections $1-\chi_{X_{B,\delta}}$ and $\Phi(\chi_B)$ have non-empty intersection. Hence, $\|\Phi(\chi_B)( 1-\chi_{X_{B,\delta}})\|=1$; contradiction.
\end{proof}

\begin{lemma}\label{LemmaExistenceInjection}
There exists $\delta>0$, an injection $g\colon Y'\to X$ and an injection $f\colon X'\to Y$ such that $g(y)\in X_{y,\delta}$ and $f(y)\in Y_{x,\delta}$ for all $y\in Y'$ and all $x\in X'$. 
\end{lemma}

\begin{proof}
By symmetry, we only show the existence of $g\colon Y'\to X$. Let $\delta$ be given by Lemma \ref{LemmaCardinalityXBDelta} for some $\eps\in (0,1)$. Define a map $\alpha\colon Y'\to \cP(X)$ by letting $\alpha(y)=X_{y,\delta}$ for all $y\in Y'$. Since $X_{B,\delta}=\cup_{y\in B}\alpha(y)$, the choice of $\delta$ gives that
\[|B|\leq |X_{B,\delta}|=\Big|\bigcup_{y\in B}\alpha(y)\Big|.\]
Therefore, by Hall's marriage theorem, the required injection exists.
\end{proof}

\begin{proof}[Proof of Theorem \ref{ThmRigidityRoeCoronaPropertyA}]
Assume \eqref{ItemLiftIso}, i.e., $X$ and $Y$ are uniformly locally finite metric spaces such that $X$ has property A, and $\roeq(X)$ and $\roeq(Y)$ are liftable on the diagonals isomorphic. Since a uniformly locally finite metric space has property A if and only if its uniform Roe algebra is nuclear (\cite[Theorem 5.3]{SkandalisTuYu2002}), $\cstu(X)$ is nulcear and so is $\roeq(X)$. Therefore, $\roeq(Y)$ is nuclear and since $\cK(\ell_2(Y))$ is nuclear, so is $\cstu(Y)$. So, $Y$ has property A. By \cite[Theorem 4.1]{Sako2014}, $X$ and $Y$ have ONL. Hence, by Propositions \ref{PropThankGodForNoNoncompactGhostProj} and \ref{X:PropCofiniteSubRank1.General} we have $\Phi$, $\Psi$, $X'$, and $Y'$ as in Assumption~\ref{Assumption.bijective} (we use here that $(\Phi(e_{xx}))_{x\in X}$ and $(\Psi(e_{yy}))_{y\in Y}$ are non-zero finite rank projections for cofinitely many elements).

Let $f\colon X'\to Y$ and $g\colon Y'\to X$ be the injections given by Lemma \ref{LemmaExistenceInjection}. By Lemma \ref{LemmaDefiMaps0} there exist cofinite subsets $X''\subseteq X'$ and $Y''\subseteq Y'$ such that 
\[\inf_{x\in X''}\|\Psi(e_{f(x)f(x)})e_{xx}\|>0\text{ and }\inf_{y\in Y''}\|\Phi(e_{g(y)g(y)})e_{yy}\|>0.\]
Since $f$ and $g$ are injective, $f^{-1}(Y\setminus Y'')$ and $g^{-1}(X\setminus X'')$ are finite. Hence, K\"onig's proof of the Cantor--Schr\"oder--Bernstein theorem gives us cofinite subsets $\tilde X\subseteq X$, $\tilde Y\subseteq Y$, and a bijection $h\colon \tilde X\to \tilde Y$ such that for each $x\in \tilde X$, either $x\in X''$ and $h(x)=f(x)$ or $x\in \mathrm{Im}(g)$ and $h(x)=g^{-1}(x)$. This implies that 
\[\|\Phi(e_{xx})e_{h(x)h(x)}\|\geq \delta\text{ and }\|\Psi(e_{yy})e_{h^{-1}(y)h^{-1}(y)}\|\geq \delta\]
for all $x\in \tilde X$ and all $y\in \tilde Y $. By Lemma \ref{LemmaTheMapsAreCoarse}, both $h$ and $h^{-1}$ are coarse. So $h$ is a coarse equivalence, and \eqref{ItemEquivv} holds.

Say \eqref{ItemEquivv} holds. Let $X'\subset X$ and $Y'\subset Y$ be cofinite subsets and $f:X'\to Y'$ be a bijective coarse equivalence. Define a linear map $U:\ell_2(X')\to \ell_2(Y')$ by letting $U\delta_x=\delta_{f(x)}$ for all $x\in X'$. Since $f$ is a coarse equivalence, it follows that the map $a\in \cstu(X')\mapsto UaU^*\in\cstu(Y')$ is an isomorphism (cf. \cite[Theorem 8.1]{BragaFarah2018Trans}). Therefore, since $X'\subset X$ and $Y'\subset Y$ are cofinite, this induces an isomorphism between $\roeq(X)$ and $\roeq(Y)$, and \eqref{ItemLiftIso} holds.
\end{proof}

\section{Independence results}\label{sec:counterexamples}
%
%
%


In this section we present two independence results, Theorem~\ref{T.Independent+} and Example~\ref{Ex:RudinShelah}. 
The results of \cite{Gha:FDD} and \cite{Ghasemi.FFV} imply the existence of countable discrete metric spaces $X$ and $Y$ such that the assertion $\roeq(X)\cong \roeq(Y)$ is independent from ZFC. 
Theorem \ref{T.Independent} is an immediate consequence of the following theorem.

\begin{theorem}\label{T.Independent+}
There exists a family $\cG$ of continuum many locally finite metric spaces such that the following holds. 
\begin{enumerate}
\item The 
Continuum Hypothesis implies $\roeq(X)\cong \roeq(Y)$ for all $X$ and $Y$ in $\cG$. 
\item $\OCA+\MA_{\aleph_1}$ implies $\roeq(X)\not\cong \roeq(Y)$ for all distinct $X$ and $Y$ in $\cG$. 
\end{enumerate}
\end{theorem}

This theorem is really a reformulation of results of Ghasemi and McKenney into the language of 
uniform Roe algebras and coronas.

A metric $d$ on $\bbN^2$ is defined as follows (writing $\bar m=(m_0,m_1)$ for an element of $\bbN^2$): 
\[
d(\bar m, \bar n)=\begin{cases} 
m_0+n_0+1, & \text{ if }m_0\neq n_0\\
1, & \text{ if $m_0=n_0$ and $m_1\neq n_1$}\\
0, & \text{ if $\bar m=\bar n$}.
\end{cases}
\]
\begin{definition} \label{Def.X(g)}
Fix $g\colon \bbN\to \bbN$, and define
\begin{align*}
X(g)=\{\bar m\colon m_1\leq g(m_0)\} \text{ and }
\cM_g=\prod_n M_{g(n)}(\bbC)/\bigoplus_n M_{g(n)}(\bbC). 
\end{align*} 
\end{definition} 
We view $X(g)$ as a subspace of $(\N^2,d)$. If $g$ is unbounded, $(X(g),d)$ is not uniformly locally finite because the $j$-th 
vertical column has cardinality $g(j)$ and diameter $1$. However, $X(g)$ is locally finite. Let
\[
\calD[g]=\{a\in \cB(\ell_2(X(g))\colon \langle a e_{\bar m}, e_{\bar n}\rangle \neq 0\Rightarrow m_0=n_0\}. 
\]

\begin{lemma}\label{L.X(g)} For every $g\colon \bbN\to \bbN$ we have $\cstu(X(g))=\cst(\calD[g], \cK(\ell_2(X(g))))$ and $\roeq(X(g))\cong \cM_g$. 
\end{lemma} 

\begin{proof}
First of all note that $\mathcal D[g]$ is a von Neumann algebra, and it is isomorphic to $\prod_n M_{g(n)}$. Also $\mathcal D[g]\cap\mathcal K(\ell_2(X(g)))=\bigoplus_nM_{g(n)}$. Therefore it is enough to prove the first assertion to get the second one. 

Since $d(\bar m,\bar n)\leq 1$ when $m_0=n_0$, we have 
 \[
 \mathcal D[g]\subseteq\{a\in\mathcal B(\ell_2(X(g)))\colon \propg(a)\leq 1\}.
 \]
 On the other hand, fix $n\in\bbN$: if $a\in\cB(\ell_2(X(g)))$ is such that $\propg(a)\leq n$, notice that $\langle ae_{\bar m},e_{\bar n}\rangle\neq 0$ implies that $m_0=n_0$ or $m_0,n_0\leq n$. With $Z_n=\{(m_1,m_2)\in X(g)\colon m_1\geq n\}$, we have that 
\[
a\chi_{Z_n}\in\mathcal D[g]\text{ and }a-a\chi_{Z_n}\in\mathcal K(\ell_2(X(g))).
\]
 This shows that $\mathcal D[g]+\mathcal K(\ell_2(X(g)))=\cstu(X)$.
\end{proof}



\begin{proof}[Proof of Theorem~\ref{T.Independent}]
By \cite[Theorem~1.2]{Ghasemi.FFV} there exists a strictly increasing function $k\colon \bbN\to \bbN$ such that whenever $g\colon \bbN\to \bbN$ is a strictly increasing function whose range is included in the range of $k$ then the Continuum Hypothesis implies $\cM_g\cong\cM_k$. 
Partition the range of $k$ into infinite sets,~$A_n$, for $n\in \bbN$. 
For $B\subseteq \bbN$ let $f_B$ the the function whose range is equal to $\bigcup_{n\in B} A_n$ and let $\cG=\{X(f_B)\colon B\subseteq \bbN\}$.

If $B$ and $B'$ are distinct subsets of $\bbN$ then the symmetric difference of the ranges of the functions $f_B$ and $f_{B'}$ is infinite, and by the equivalence of (1) and (2) of \cite[Theorem~1.2]{Ghasemi.FFV} the coronas $\cM_{f_B}$ ad $\cM_{f_{B'}}$ are not isomorphic in the model constructed in \cite[Corollary~1.1]{Gha:FDD}, or in any model of ZFC in which $\OCA+\MA_{\aleph_1}$ holds (see \cite[Corollary 1.7]{McKenney.UHF}). 
\end{proof} 

All spaces in $\cG$ are coarsely equivalent to the subspace $\{ n^2\colon n\in \bbN\}$ of~$\bbN$. 

We conclude by showing how some well-known results directly translate into an independence result about uniform Roe coronas. 

\begin{exmpl} \label{Ex:RudinShelah}
There exists a uniformly locally finite metric space $X$ such that the assertion `Every automorphism of $\roeq(X)$ is liftable on diagonals' is independent from ZFC. 

Let $X$ be $\{n^2\colon n\in \bbN\}$ with the metric inherited from $\bbN$. Then $\cstu(X)$ is the algebra generated by $\ell_\infty(X)$ and $\cK(\ell_2(X))$, and $\roeq(X)\cong \ell_\infty(\bbN)/c_0(\bbN)$. 
This is the abelian \cstar-algebra whose spectrum is homeomorphic to the \v Cech--Stone remainder (corona) of $\bbN$, $\beta\bbN\setminus \bbN$. 
An automorphism $\Phi$ of $\ell_\infty(X) /c_0(X)$ is liftable on diagonals if and only if it has a lift which is a $^*$-homomorphism. 
This is equivalent to asserting that the dual map $\Phi_*\colon \beta\bbN\setminus \bbN\to \beta\bbN\setminus \bbN$ has a continuous extension to a map from $\beta\bbN$ to $ \beta\bbN$. Since every such map is determined by its restriction to $\bbN$, there are only $2^{\aleph_0}$ such (so-called `trivial') automorphisms of $\ell_\infty(\bbN)/c_0(\bbN)$. 

It remains to see that the assertion `all autohomeomorphisms of $\beta\bbN\setminus \bbN$ are trivial' is independent from ZFC. 
In \cite{Ru} W. Rudin used the Continuum Hypothesis to construct $2^{2^{\aleph_0}}$ nontrivial automorphisms of $\beta\bbN\setminus \bbN$ and 
 S. Shelah proved that is relatively consistent with ZFC that all automorphisms are trivial (\cite{Sh:PIF}). 
 \end{exmpl} 
 
 Two additional (and related) curiosities about the uniform Roe corona from Example~\ref{Ex:RudinShelah} are worth noting. 
First, it has a unital embedding into itself that is almost liftable on the diagonal, but not liftable on the diagonal: take the special case of \cite[Example~3.2.1]{FarahBook2000} where $\SI=\SJ_1=\SJ_2=\Fin$, the ideal of finite subsets of $\bbN$. 
This gives an injective endomorphism of the Boolean algebra $\cPNF$. By combining Stone duality with the Gelfand--Naimark duality, one obtains an embedding as required. 
An example of a surjective $^*$-homomorphism $\ell_\infty/c_0\to\ell_\infty/c_0$ that (in our terminology) cannot be lifted by a $^*$-homomorphism was constructed (in ZFC!) in~\cite{dow2014non}.

 \section{Concluding remarks}

 We conclude by stating some problems related to our work that remain open. 
 
 The coarse Baum--Connes conjecture is directly related to the variant of the 
 uniform Roe algebra $\cstu(X)$ called the \emph{Roe algebra}, $\cst(X)$. 
 The algebraic Roe algebras and the Roe algebras are algebras of operators on $\ell_2(X,H)$ for an infinite-dimensional separable Hilbert space $H$. They are defined like their uniform 
 analogs, but the matrix entries $e_{xx'}$ are allowed to be arbitrary compact operators on $H$ (see e.g., \cite[\S 2]{SpakulaWillett2013AdvMath}). 
 Every Roe algebra $\cst(X)$ contains the compact operators, and the \emph{Roe corona}, $\mathrm{Q}^*(X)$, is the
 quotient of $\cst(X)$ over the ideal of all compact operators.

 \begin{problem}\emph{\textbf{(Rigidity of Roe Coronas)}}
 	Let $X$ and $Y$ be metric spaces such that $\mathrm{Q}^*(X)$ and $\mathrm{Q}^*(Y)$ are isomorphic. Does it follow that $X$ and $Y$ are coarsely equivalent? 	
 \end{problem}
 
The second natural occurring problem is the one of lifting isomorphisms between uniform Roe coronas to isomorphisms of uniform Roe algebras. It is not difficult to find an example of uniformly locally finite metric spaces such that all of their sparse subspaces yield only compact ghost projections and $\roeq(X)$ and $\roeq(Y)$ are liftable on diagonals isomorphic, but $\cstu(X)$ and $\cstu(Y)$ are not isomorphic. In particular, the isomorphism cannot be lifted by an isomoprhism between $\cstu(X)$ and $\cstu(Y)$. 
 However, if $X'\subseteq X$ is a cofinite subspace of $X$ then $\roeq(X)$ and $\roeq(X')$ can be 
 naturally identified, and the counterexample alluded to in this paragraph does not answer the following question.

 \begin{question}\label{Q.lift} 
 	Suppose $X$ and $Y$ are uniformly locally finite metric spaces such that all of their sparse subspaces yield only compact ghost projections and $\Lambda\colon \roeq(X)\to\roeq(Y)$ is an isomorphism liftable on diagonals. 
 	Are there cofinite $X'\subseteq X$ and $Y'\subseteq Y$ such that $\Lambda$ can be lifted to an isomorphism 
 	between $\cstu(X')$ and $\cstu(Y')$? 
 \end{question} 
 
 
 A natural line of attack on this problem leads to a noncommutative variant of the concept of a (balanced) near action extensively studied in \cite{cornulier2019near}. A \emph{near action} of a group $\Gamma$ on a set~$X$ is a group homomorphism from $\Gamma$ into $\fS^*(X)$, defined to be the quotient of the group $S_X$ of  all permutations of $X$ modulo the normal subgroup of finitely supported permutations. A near action is said to be \emph{realizable} if it can be lifted to a group homomorphism from $\Gamma$ into $S_X$ (\cite[Definition~4.A.4]{cornulier2019near}).\footnote{The requirement of being balanced is, in this context, naive (see the discussion in  \cite[\S 1.B and \S 1.D]{cornulier2019near}), but dropping it in our case would only lead to unnecessary complications.}  
 
 A \emph{noncommutative near action} of a group $\Gamma$ on a set $X$ is a group homomorphism~$F\colon \Gamma\to \cN(\ell_\infty(X)/c_0(X))$, the normalizer of the masa $\ell_\infty(X)/c_0(X)$ in the Calkin algebra $\cQ(\ell_2(X))$. It is \emph{balanced} if  $F(g)$ has Fredholm index zero (i.e., can be lifted to a unitary in $\mathcal B(\ell_2(X))$) for every $g\in \Gamma$.  It is \emph{realizable} if it can be lifted to a homomorphism from $\Gamma$ into the unitary group of $\cN(\ell_\infty(X))$, the normalizer of $\ell_\infty(X)$ in $\mathcal B(\ell_2(X))$. Clearly, being balanced is a necessary condition for realizability. Suppose that $\Lambda\colon \roeq(\Gamma)\to \roeq(\Gamma')$ is an isomorphism liftable on the diagonal by $\Phi\colon\ell_\infty(\Gamma)\to\cstu(\Gamma')$. Let $\iota\colon \roeq(\Gamma')\to \cQ(\ell_2(\Gamma'))$ be the identity embedding. Recall that the canonical unitaries in $\cstu(\Gamma)$ are those of the form $u_g$, for $g\in \Gamma$, where $u_g\delta_h=\delta_{gh}$. Each of these has Fredholm index 0. Also, for a unitary in the normalizer of $\ell_\infty(\Gamma)$, having Fredholm index zero is coded by its conjugation action on the masa. Therefore   $\iota\circ \Lambda$ gives a noncommutative balanced near action $F_\Lambda$ of $\Gamma$ on a countable set $X$, where $X=\{\Phi(\chi_{\{g\}}), g\in\Gamma\}$ is the countable set associated with the lift of the image of $\ell_\infty(\Gamma)$, so that $\Phi[\ell_\infty(\Gamma)]=\ell_\infty(X)$. 
 
 \begin{lemma} \label{L.near}
 Suppose that $\Lambda\colon \roeq(\Gamma)\to \roeq(\Gamma')$ is an isomorphism liftable on the diagonal, and let $F_\Lambda$ be the associated noncommutative balanced near action as above. Then $F_\Lambda$ is realized if and only if $\Lambda$ can be lifted by a $^*$-homomorphism. 
  \end{lemma}

\begin{proof} Let $\Phi\colon \ell_\infty(\Gamma)\to\cstu(\Gamma')$    be the lift of $\Lambda$ on the diagonal and let $X$ be as in the paragraph preceding the lemma. So, $\Phi:\ell_\infty(\Gamma)\to \ell_\infty(X)$. 

Suppose that $\Lambda$ can be lifted by a $^*$-homomorphism $\Psi\colon \cstu(\Gamma)\to \cstu(\Gamma')$. Since any two atomic masas in $\cQ(\ell_2(\Gamma'))$ are unitarily equivalent, we can fix a unitary $w$ in $\cB(\ell_2(\Gamma'))$ such that $\Ad w\circ \Psi$ and $\Phi$ agree on $\ell_\infty(\Gamma)$. Then (denoting the canonical unitary in $\cstu(\Gamma)$ associated with $g\in \Gamma$ by $u_g$) the group homomorphism $g\mapsto \Ad w(\Psi(u_g))$ realizes $F_\Lambda$.

For the converse direction, assume that $g\mapsto u_g$ realizes $F_\Lambda$. We need to verify that this group homomorphism is compatible with $\Phi$. For $a\in \ell_\infty(\Gamma)$ we have $u_g a u_g^*=g.a$, where $g.a(\delta_h)=a(\delta_{g^{-1}h})$. Also,  
\[
v_g \Phi(a) v_g^*= \Phi(g.a).
\] 
	The algebraic uniform Roe algebra $\cstu[\Gamma]$ consists of sums of the form $\sum_{g\in \Gamma} a_g u_g$, where $a_g\in \ell_\infty(\Gamma)$ and only finitely many of $a_g$ are nonzero. For such sum, define
	\[
	\textstyle\Psi(\sum_g a_g u_g)=\sum_g \Phi(a_g) v_g. 
	\]
	The fact that $g\mapsto v_g$ is a group homomorphism, together with the previous displayed formula and routine calculations, implies that $\Psi$ defines a $^*$-homomorphism. Clearly, $\Lambda(\pi_\Gamma(a))=\pi_{\Gamma'}(\Psi(a))$ for all $a\in \cstu[\Gamma]$, and we can extend $\Psi$ to a lift of $\Lambda$. 
\end{proof}

 By the freeness property, every noncommutative balanced near action of a free group is realizable. This observation, together with Lemma~\ref{L.near}, implies that Question~\ref{Q.lift} has a positive answer in the case when $X$ is the free group~$F_n$ with the Cayley graph metric for some finite $n$.  
 
 The following observation is taken from \cite{deCornulier.mathoverflow}. 
 
 \begin{example} \label{Ex.Z2} There exists a nonrealizable noncommutative balanced near action of $\bbZ^2$ on $\bbN$.  We will find unitaries $u_a$ and $u_b$ in $\cN(\ell_\infty(\bbZ^2))$ such that $u_a u_b -u_b u_a$ is compact, but there are no commuting unitaries $v_a$ and $v_b$ in $\cN(\ell_\infty(\bbZ^2))$ such that both $u_a-v_a$ and $u_b-v_b$ are compact. 
 Define $u_a$ by its action on the basis by $u_a \delta_{(m,n)}=\delta_{(m,n+1)}$ for all $(m,n)\in \bbZ^2$. Similarly, define $u_b$ by $u_b \delta_{(m,n)}=\delta_{(m+1,n)}$ if $m\leq |n|$ and 
 \[
 u_b \delta_{(m,n)}=\exp( \pi i(m-|n|)/m)\delta_{(m+1,n)}
 \]
 if $m>|n|$. It is straightforward to check that $u_a u_b -u_b u_a$ is compact. 
 	A moment of thought reveals that for  every $m$ and $k\geq 0$ we have 
 \[
 (u_b)^k\delta_{(m,0)}=(-1)^k \delta_{(m+k,0)}
 \]
 and yet
 \[
 (u_a)^{-k}(u_b u_a)^k \delta_{(m,0)}=\delta_{(m+k,0)}. 
 \]
This readily implies that if $v_a$ and $v_b$ are unitaries in $\cN(\ell_\infty(\bbZ^2))$ such that both $u_a-v_a$ and $u_b-v_b$ are compact, then $v_a$ and $v_b$ cannot be commuting. 

The required near action is defined by sending the generators of $\bbZ^2$ to $\pi(u_a)$ and $\pi(u_b)$.  
  \end{example}
 
 We emphasize that Example~\ref{Ex.Z2} does not imply that the answer to Question~\ref{Q.lift} in case when $X$ is $\bbZ^2$  with the Cayley graph metric is negative. Indeed, this is the case since we do not know whether the near action of Example~\ref{Ex.Z2} is of the form $F_\Lambda$ for some appropriate $\Lambda$.
  
Lastly, we focus on generalizations of Theorem~\ref{T.Independent} in presence of uniform local finiteness.
 
 \begin{question}Are there uniformly locally finite spaces $X$ and $Y$ such that the existence of an isomorphism between their uniform Roe coronas is independent from ZFC? 
 \end{question}
 
 For any two countable metric spaces $X$ and $Y$, the assertion `$X$ is coarsely equivalent to $Y$' is unlikely to be independent from ZFC. This is because it is equivalent to the assertion that there are $f\colon X\to Y$ and $g\colon Y\to X$ that satisfy certain first-order conditions. A statement of this form (so-called $\mathbf\Sigma^1_1$ statement) has the same 
 truth value in all transitive models of ZFC (see e.g., \cite[Theorem~13.15]{Kana:Book} for a stronger result). 
 Strictly speaking, this does not imply that the assertion `$X$ is coarsely equivalent to $Y$' cannot be undecidable; 
 instead it shows that the conventional set-theoretic methods cannot show its undecidability.

\section*{Acknowledgments}
The authors would like to thank Rufus Willett for many useful conversations. The first two authors would also like to thank the Fields Institute 
for allowing them to make use of its facilities. 
This project started during the very stimulating 
`Approximation Properties in Operator Algebras and Ergodic Theory'
workshop at UCLA/IPAM and continued during the third author's visit to the Fields Institute in the Spring of 2018. IF would like to thank 
Hanfeng Li and Wouter van Limbeek for stimulating conversations and 
the organizers of the workshop for the invitation. IF is partially supported by NSERC. AV was supported by a PRESTIGE co-fund grant and a FWO scholarship, and now is supported by an Emergence en Recherche grant and partially supported.by the ANR project Agrume (ANR-17-CE40-0026).

\end{document}